\documentclass[11pt]{article}
\usepackage[utf8]{inputenc}

\usepackage{
amsmath,amssymb,amsfonts,amsthm}
\usepackage{bm}
\usepackage{colonequals}
\usepackage{comment}
\usepackage{epsfig} \usepackage{latexsym,nicefrac,bbm}
\usepackage
{xspace}
\usepackage{color,fancybox,graphicx,subfigure,fullpage}
\usepackage[top=1.25in, bottom=1.25in, left=1in, right=1in]{geometry}
\usepackage{tabularx}
\usepackage{hyperref}
\usepackage{cleveref}
\usepackage{pdfsync}
\usepackage[boxruled,linesnumbered]{algorithm2e}

\usepackage{multicol}
\usepackage{enumitem}
\usepackage{float}
\usepackage{tikz}

\newtheorem{theorem}{Theorem}[section]

\newtheorem{definition}[theorem]{Definition}
\newtheorem{lemma}[theorem]{Lemma}
\newtheorem{remark}[theorem]{Remark}
\newtheorem{proposition}[theorem]{Proposition}
\newtheorem{corollary}{Corollary}[section]

\newtheorem{question}[theorem]{Question}

\newcommand{\CC}{\mathbb{C}}

\newcommand{\EE}{\mathbb{E}}

\newcommand{\NN}{\mathbb{N}}

\newcommand{\RR}{\mathbb{R}}

\DeclareSymbolFont{bbold}{U}{bbold}{m}{n}
\DeclareSymbolFontAlphabet{\mathbbold}{bbold}

\newcommand{\Var}{\mathrm{Var}}

\renewcommand{\emptyset}{\varnothing}

\renewcommand{\epsilon}{\varepsilon}
\renewcommand{\Tilde}{\widetilde}
\renewcommand{\hat}{\widehat}

\newcommand{\aln}{\mathrm{align}}
\newcommand{\Ex}{\mathop{\mathbb{E}}}
\newcommand{\Px}{\mathop{\mathbb{P}}}
\newcommand{\cc}{\mathrm{cc}}
\newcommand{\dbar}{\,\|\,}

\newcommand{\sQ}{\mathcal{Q}}
\newcommand{\sP}{\mathcal{P}}
\newcommand{\eye}{\bm i}

\newcommand\numberthis{\addtocounter{equation}{1}\tag{\theequation}}

\newif\ifnotes
\notestrue
\newcommand{\authnote}[3]{\textcolor{#3}{[{{\bf #1:} { {#2}}}]}}
\newcommand{\todo}[1]{\ifnotes \authnote{TODO}{#1}{red} \fi}

\title{Inference of rankings planted in random tournaments}

\usepackage{authblk}

\author{Dmitriy Kunisky\thanks{Email: \texttt{dmitriy.kunisky@yale.edu}. Supported in part by ONR Award N00014-20-1-2335, a Simons Investigator Award to Daniel Spielman, and NSF grants DMS-1712730 and DMS-1719545.}}

\author{Daniel A.\ Spielman\thanks{Email: \texttt{daniel.spielman@yale.edu}. Supported in part by ONR Award N00014-20-1-2335 and a Simons Investigator Award to Daniel Spielman.}}

\author{Xifan Yu\thanks{Email: \texttt{xifan.yu@yale.edu}. Supported in part by a Simons Investigator Award to Daniel Spielman.}}

\affil{Department of Computer Science, Yale University}

\date{\today}

\begin{document}

\maketitle

\thispagestyle{empty}

\begin{abstract}
    We consider the problem of inferring an unknown ranking of $n$ items from a random tournament on $n$ vertices whose edge directions are correlated with the ranking.
    We establish, in terms of the strength of these correlations, the computational and statistical thresholds for detection (deciding whether an observed tournament is purely random or drawn correlated with a hidden ranking) and recovery (estimating the hidden ranking with small error in Spearman's footrule or Kendall's tau metric on permutations).
    Notably, we find that this problem provides a new instance of a \emph{detection-recovery gap}: solving the detection problem requires much weaker correlations than solving the recovery problem.
    In establishing these thresholds, we also identify simple algorithms for detection (thresholding a degree~2 polynomial) and recovery (outputting a ranking by the number of ``wins'' of a tournament vertex, i.e., the out-degree) that achieve optimal performance up to constants in the correlation strength.
    For detection, we find that the above low-degree polynomial algorithm is superior to a natural spectral algorithm.
    We also find that, whenever it is possible to achieve strong recovery (i.e., to estimate with vanishing error in the above metrics) of the hidden ranking, then the above ``Ranking By Wins'' algorithm not only does so, but also outputs a close approximation of the maximum likelihood estimator, a task that is NP-hard in the worst case.
\end{abstract}

\clearpage

\pagestyle{empty}

\tableofcontents

\clearpage

\setcounter{page}{1}
\pagestyle{plain}

\section{Introduction}
We study three problems in statistics and optimization concerning a model that generates noisy observations of pairwise comparisons of $n$ items.
In our model, there is a ``hidden'' permutation $\pi$ that encodes the ``true'' ranking of the $n$ items, but we are only given noisy observations of whether $\pi(i) < \pi(j)$ or vice-versa.
Concretely, for each $i, j \in [n]$ distinct, we observe $T_{i,j} = -T_{j, i} \in \{\pm 1\}$.
These labels are generated at random by sampling, for each pair of $i, j$ that satisfies $\pi(i) < \pi(j)$,
\begin{equation}
    T_{i,j} = -T_{j, i} = \begin{cases}
        +1 & \quad \text{ with probability } \frac{1}{2} + \gamma, \\
        -1 & \quad \text{ with probability } \frac{1}{2} - \gamma,
    \end{cases}
\end{equation}
for some $\gamma \in [0, \frac{1}{2}]$.
In words, when $\pi(i) < \pi(j)$ then $T_{i,j}$ is biased towards $+1$, while when $\pi(i) > \pi(j)$ then $T_{i,j}$ is biased towards $-1$.

It is instructive to consider the extreme values of $\gamma$: if $\gamma = \frac{1}{2}$, then $q$ tells us directly each pairwise comparison between $\pi(i)$ and $\pi(j)$, and therefore tells us $\pi$ itself, while if $\gamma = 0$, then we observe a uniformly random tournament that contains no information concerning $\pi$.
For $\gamma$ in between, we have partial but noisy information about the underlying ranking $\pi$, and we are interested in how informative our observations are.

From now on, we will refer to the above model as the \emph{planted ranking model}, or the \emph{planted model} for short, and $\gamma > 0$ as the \emph{signal} parameter.
We denote this model by $\sP$.
We will refer to the case $\gamma = 0$, in which case the model generates a uniformly random $\pm 1$-valued skew-symmetric matrix $T$, as the \emph{null model}, denoted $\sQ$.

In this paper, we study three fundamental problems in this setting:
\begin{enumerate}
    \item For what $\gamma$ can we distinguish whether $T$ was drawn from the planted model or from the null model?
    \item For what $\gamma$ can we approximately recover the hidden permutation $\pi$ in the planted model?
    \item For what $\gamma$ can we find an approximate maximum likelihood estimator (MLE) of $\pi$ under the planted model? As we will see in Section~\ref{sec:mle}, this will amount to finding an approximate maximizer $\hat{\pi}$ of the objective function
    \begin{equation}
    \aln(\hat{\pi}, T) \colonequals \sum_{(i,j): T_{i, j} = 1} \big(\boldsymbol{1}\{\hat{\pi}(i) < \hat{\pi}(j)\} - \boldsymbol{1}\{\hat{\pi}(i) > \hat{\pi}(j)\}\big),
    \end{equation}
    which we call the \emph{alignment objective}.
\end{enumerate}
We will refer to the problems above as the \emph{detection} task, the \emph{recovery} task, and the \emph{alignment maximization} task, respectively.
The first two questions each admit two variants: we may ask about \emph{computational} feasibility---whether these tasks can be achieved by polynomial-time algorithms---or about \emph{information-theoretic} feasibility---whether they can be achieved by any computations on $T$ at all.
For the last question, on the other hand, we are only interested in efficient algorithms, since the MLE can of course be computed inefficiently by brute force (in our case in time $O(n!)$, by brute force enumeration of all permutations).

Let us make a few remarks on our terminology for the objects $\pi$ and $T$.
We identify $\pi$ with the corresponding \emph{ranking} or total ordering of $[n]$ given by $\pi^{-1}(1) <_{\pi} \pi^{-1}(2) <_{\pi} \dots <_{\pi} \pi^{-1}(n) $, where $\pi^{-1}(i)$ is the element that is ranked $i$-th and $<_{\pi}$ denotes the induced ordering on $[n]$.
Thus, $i <_{\pi} j$ if and only if $\pi(i) < \pi(j)$, and in the planted model we have $T_{i, j}$ biased towards 1 in this case.

We also identify $T$ with the corresponding \emph{tournament}, an assignment of directions to each edge of the complete graph on $[n]$, where we include the directed edge $(i, j)$ (whose tail is $i$ and whose head is $j$) if $T_{i, j} = 1$, and include $(j, i)$ otherwise.
If we want to distinguish the tournament from this $\pm 1$-valued skew-symmetric adjacency matrix, we call the matrix the \emph{tournament matrix}.
We also identify a permutation $\pi$ with the associated tournament describing the total ordering $<_{\pi}$,
\begin{equation}
\pi(i, j) \colonequals
\begin{cases}
    +1 & \quad \text{ if } i <_{\pi} j,\\
    -1 & \quad \text{ if } i >_{\pi} j,
\end{cases}
\end{equation}
which will turn out to be convenient later on.
For instance, with this notation, we may rewrite the alignment objective more simply as
\begin{equation}
    \aln(\hat{\pi}, T) = \frac{1}{2}\sum_{i, j} T_{i, j} \hat{\pi}(i, j) = \sum_{i < j} T_{i, j} \hat{\pi}(i, j).
\end{equation}

\subsection{Main Contributions}

We are now ready to summarize our main contributions.
These will describe the feasibility of the three tasks mentioned above, pinpointing the scalings of $\gamma = \gamma(n)$ at which they become possible or impossible, for either polynomial-time or arbitrary computations (i.e., identifying computational and information-theoretic thresholds, respectively).
\begin{enumerate}
\item For the detection task, we show a transition on the order of $\gamma \sim n^{-3/4}$. When $\gamma = \omega(n^{-3/4})$, \emph{strong detection} between the planted model and the null model (i.e., identifying with high probability as $n \to \infty$ whether $T \sim \sP$ or $T \sim \sQ$) is achieved by thresholding a simple degree $2$ polynomial.
When $\gamma = O(n^{-3/4})$, it is information-theoretically impossible to achieve strong detection.
On the other hand, there is a constant $c > 0$ such that, if $\gamma \ge c \cdot n^{-3/4}$, then thresholding the aforementioned degree $2$ polynomial achieves \emph{weak detection} (i.e., identifies with probability at least $\frac{1}{2} + \delta$ for some $\delta > 0$ whether $T \sim \sP$ or $T \sim \sQ$).
Finally, when $\gamma = o(n^{-3/4})$, it is information-theoretically impossible to achieve even weak detection.
In contrast, we show that a natural spectral algorithm transitions from succeeding to failing to achieve detection on the suboptimal larger order of $\gamma \sim n^{-1/2}$.

\item For the recovery task, we show a transition on the order of $\gamma \sim n^{-1/2}$. When $\gamma = \omega(n^{-1/2})$, \emph{strong recovery} (estimating $\pi$ with vanishing error as $n \to \infty$) is achieved by a simple ``Ranking By Wins'' algorithm, and when $\gamma = \Theta(n^{-1/2})$, \emph{weak recovery} (estimating $\pi$ with error less than that achieved by a random guess) is achieved by the same algorithm, but strong recovery is information-theoretically impossible. Finally, when $\gamma = o(n^{-1/2})$, it is information-theoretically impossible to achieve even weak recovery.

\item For the alignment maximization problem, we show that the same Ranking By Wins algorithm with high probability constructs a $(1 - o(1))$-approximation to the maximum alignment on input from the planted model when $\gamma = \omega(n^{-1/2})$.
\end{enumerate}

We emphasize that the first two results above imply that this problem provides a new example of a \emph{detection-recovery gap}, a phenomenon of much recent interest in the literature on algorithmic high-dimensional statistics.
For many problems, it was previously observed that the thresholds for strong detection and weak recovery coincide; for instance, \cite{Abbe-2017-SBMReview} discusses this point concerning the stochastic block model and its variations.
More recently, natural examples of problems were found where this does not occur, for instance for the planted dense subgraph \cite{CX-2016-ThresholdsPlantedClustersGrowing,SW-2020-LowDegreeEstimation,BJ-2023-DetectionRecoveryGapPDS} and planted dense cycle \cite{MWZ-2023-DetectionRecoveryPlantedCycles} problems.
Still, this phenomenon remains mysterious, and we hope that this new example may help to clarify the origins of detection-recovery gaps.

\subsection{Related Works}

Our model of random tournaments has been studied in the past, but mostly from the perspective of particular algorithms seeking to compute an approximation to the MLE of $\pi$, i.e., to approximately solve the alignment maximization task.

In \cite{braverman2007noisy}, an efficient algorithm was proposed that that with high probability (w.h.p.) \emph{exactly} computes the MLE of the hidden permutation, for the signal scaling $\gamma = \Theta(1)$.\footnote{We use the term \emph{with high probability} for sequences of events occurring with probability converging to 1 as $n \to \infty$.}
Moreover, it is shown that the MLE is close to the hidden permutation, with a total dislocation $O(n)$ and the maximum dislocation $O(\log n)$.\footnote{The \emph{dislocation} of $i$ between permutations $\pi$ and $\rho$ is $|\pi(i) - \rho(i)|$, the total dislocation is the sum of these over all $i$---which is the same as the Spearman's footrule distance---and the maximum dislocation is the maximum of these over all $i$.} A faster $O(n^2)$-time algorithm is given in \cite{klein2011tolerant} in the same setting as \cite{braverman2007noisy}, but that algorithm does not output the exact MLE and has a worse guarantee on the total dislocation distance.
As our results show, the scaling $\gamma = \Theta(1)$ is also far greater than the thresholds for efficiently recovering or detecting a hidden ranking with other algorithms.

Improving on this scaling, \cite{rubinstein2017sorting} gave an efficient algorithm that again w.h.p~exactly computes the MLE, now for $\gamma = \Omega((\log \log n / \log n)^{1/6})$.
The sequence of works \cite{geissmann2017sorting, geissmann2018optimal} yielded an algorithm that achieves the same approximation guarantee as in \cite{braverman2007noisy} with an improved running time of $O(n \log n)$, but that again does not compute the exact MLE and operates under an even more stringent assumption that $\gamma > 7/16$ is a sufficiently large constant.

As mentioned earlier, any set of pairwise comparisons of $n$ items may be viewed as a tournament on $n$ vertices, or an assignment of directions to the edges of a complete graph.
Finding the MLE under the planted model corresponds to the well-known \emph{feedback arc set} problem on directed graphs, which is known to be NP-hard \cite{ailon2008aggregating, alon2006ranking}.
This problem, given a tournament, asks for a permutation of the vertices that maximizes the number of directed edges of the tournament that are consistent with the permutation (i.e., whose tail vertex is smaller under the permutation than the head vertex).
The literature on this problem is concerned with a different normalization of the alignment objective, however, given, for a tournament $T$, by
\begin{align*}
    \mathrm{obj}(\hat{\pi}, T) \colonequals \sum_{(i,j)\in E(T)} \boldsymbol{1}\{\hat{\pi}(i) < \hat{\pi}(j)\} = \frac{1}{2}\left(\binom{n}{2} + \aln(\hat{\pi}, T)\right).
\end{align*}
A PTAS for maximizing $\mathrm{obj}(\hat{\pi}, T)$ with running time doubly exponential in the error $\varepsilon$ is given by \cite{kenyon2007rank}.
We note that while $\mathrm{obj}(\hat{\pi}, T)$ above takes value between $0$ and $\binom{n}{2}$, the optimum is always at least $\frac{1}{2} \binom{n}{2}$.
In fact, \cite{spencer1971optimal, de1983maximum} prove that the optimum is always at least $\frac{1}{2}\binom{n}{2} + \Omega(n^{3/2})$ and such a permutation can be constructed in polynomial time \cite{poljak1988tournament}.
Therefore, a constant multiplicative error $\varepsilon > 0$ in the PTAS of \cite{kenyon2007rank} translates to an additive error of $\frac{1}{2}\binom{n}{2} \varepsilon$, which is potentially much larger than $n^{3/2}$ and might not capture the actual magnitude of deviation from the ``base value'' of $\frac{1}{2}\binom{n}{2}$.

For this reason we instead work with $\aln(\hat{\pi}, T)$ in this paper, which is equal to the difference between the numbers of consistent and inconsistent edges in the tournament $T$ with respect to a given ranking $\hat{\pi}$.
Our alignment objective function has its value between $\Omega(n^{3/2})$ and $n^2$. In particular, the algorithm in \cite{kenyon2007rank} is no longer a PTAS for this renormalized objective.
On the other hand, our results show that the Ranking By Wins algorithm is a good approximation algorithm for this objective, albeit only for a specific class of average-case inputs distributed according to the planted model with sufficiently strong signal.

Another line of research in this direction is concerned with the sample complexity of learning the hidden permutation, where the noisy comparison between one pair of items may potentially be queried multiple times, each time returning a fresh noisy label. \cite{wauthier2013efficient} gives algorithms in the \emph{noiseless} setting that, for any $\varepsilon > 0$, queries $O(n)$ pairs of comparisons and recovers a permutation at Kendall tau distance $\varepsilon \binom{n}{2}$ in expectation from the hidden permutation. In the regime when $\gamma > 0$ is a constant, one may employ an $O(n \log n)$-time sorting algorithm, and for each comparison query the noisy label $O(\log n)$ times to get the correct comparison with high probability, which gives an algorithm with $O(n \log^2 n)$ sample complexity. More sophisticated methods \cite{feige1990computing, karp2007noisy} in fact show that the hidden permutation can be found w.h.p.~with $O(n \log(n / \gamma))$ sample complexity, even if $\gamma > 0$ is not a constant. When $\gamma > 0$ is a constant, the results in \cite{wang2022noisy}, further improved by \cite{gu2023optimal}, give tight bounds on the number of noisy comparisons needed to recover the hidden permutation.

We also mention another popular model that generates noisy pairwise comparisons between $n$ elements, the Bradley-Terry-Luce (BTL) model, which was introduced in \cite{bradley1952rank, luce1959individual}. In the BTL model, there is a hidden \emph{preference vector} $w = (w_1, \dots, w_n) \in \mathbb{R}_{> 0}^n$, such that one observes a noisy label $T_{i, j}$ that takes value $+1$ with probability $w_i / (w_i + w_j)$, and $-1$ with probability $w_j / (w_i + w_j)$.
As above, such models are usually studied in terms of query complexity, with multiple independent queries of the same pair $(i, j)$ allowed.
There have been extensive studies of when one can approximate the preference vector $w$ (see e.g.~\cite{negahban2012iterative, rajkumar2014statistical, shah2018simple}) or recover the top $k$ elements (see e.g.~\cite{jang2016top, chen2015spectral, mohajer2017active}) in the BTL model.
But, the BTL model is quite different from ours, because the magnitudes of the $w_i$ can create a broader range of biases in the observations than our single parameter $\gamma$.

\section{Main Results}

\subsection{Notations}

We will use $[n]$ to denote $\{1, 2, \dots, n\}$.
For a set $A$, $\binom{A}{k}$ denotes the collection of subsets of $A$ of size $k$.
For two sets $A, B$, $A \triangle B$ denotes the symmetric difference of $A$ and $B$.

A directed graph is a pair $(V, E)$ of a vertex set $V$ and a set $E$ of directed edges, which are ordered pairs of distinct vertices.
We denote a directed edge from vertex $i$ to vertex $j$ as $(i, j)$.

We will use standard asymptotic notations $O(\cdot), o(\cdot), \Omega(\cdot), \omega(\cdot), \Theta(\cdot)$, which will always refer to the limit $n \to \infty$.
We will additionally use $\Tilde{O}$ to omit multiplicative factors depending polylogarithmically on $n$.
We use the term \emph{with high probability (w.h.p.)} to refer to sequences of events that occur with probability $1 - o(1)$ as $n \to \infty$.


We define a \emph{skewed Rademacher} random variable $\text{Rad}(p)$ as a random variable that takes value $+1$ with probability $p$, and $-1$ with probability $1-p$.
We write $d_{\mathrm{TV}}(\cdot, \cdot)$ for the total variation distance between two probability measures, $D_{\mathrm{KL}}(\cdot, \cdot)$ for the Kullback-Leibler divergence, and $\chi^2(\cdot \dbar \cdot)$ for the $\chi^2$-divergence.

\subsection{Detection}

We will consider the following notion of an algorithm achieving detection between $\mathcal{P}$ and $\mathcal{Q}$.
\begin{definition}[Weak and strong detection]
    An algorithm $A$ that takes as input a tournament $T$ on $n$ vertices and outputs an element of $\{0,1\}$ is said to achieve \emph{weak detection} between $\mathcal{P}$ and $\mathcal{Q}$ if there exists a constant $\delta > 0$ such that
    \begin{align*}
        \limsup_{n \to \infty} \left(\Px_{T \sim \mathcal{Q}}[A(T) = 1] + \Px_{T \sim \mathcal{P}}[A(T) = 0]\right) \le 1 - \delta.
    \end{align*}
    It is said to achieve \emph{strong detection} between $\mathcal{P}$ and $\mathcal{Q}$ if
    \begin{align*}
        \lim_{n \to \infty} \left(\Px_{T \sim \mathcal{Q}}[A(T) = 1] + \Px_{T \sim \mathcal{P}}[A(T) = 0]\right) = 0.
    \end{align*}
\end{definition}
\noindent
With this precise definition, let us now also restate our results on detection more carefully.

\begin{theorem}[Detection thresholds]\label{thm:detection-th}
    The following hold:
    \begin{enumerate}
    \item If $\gamma = \omega(n^{-3/4})$, then there exists a polynomial-time algorithm that achieves strong detection between $\mathcal{P}$ and $\mathcal{Q}$.

    \item If $\gamma = O(n^{-3/4})$, then exists no algorithm achieving strong detection between $\sP$ and $\sQ$ (i.e., strong detection is information-theoretically impossible).

    \item There exists a constant $c > 0$ such that, if $\gamma \ge c \cdot n^{-3/4}$, then there exists a polynomial-time algorithm that achieves weak detection between $\mathcal{P}$ and $\mathcal{Q}$.

    \item If $\gamma = o(n^{-3/4})$, then there exists no algorithm achieving weak detection between $\sP$ and $\sQ$ (i.e., weak detection is information-theoretically impossible).
    \end{enumerate}
\end{theorem}

The basic idea of the proofs of these statements is as follows.
For the negative results on information-theoretic impossibility, we bound the $\chi^2$-divergence between $\sP$ and $\sQ$ (which in turn controls the total variation distance and allows us show that detection algorithms cannot exist).
We bound this by Boolean Fourier analysis, expanding the likelihood ratio between $\sP$ and $\sQ$ in the Fourier basis.
This expansion also directs the proofs of the positive results: considering the low-degree Fourier modes of the likelihood ratio suggests a natural choice of a low-degree polynomial whose value should distinguish $\sP$ and $\sQ$, and we verify that this is indeed the case.

As an additional point of comparison, we give the following analysis of a natural spectral algorithm for the detection task.
As we detail in Section~\ref{sec:spectral}, if $T$ is a tournament matrix, then, for $\eye$ the imaginary unit, $\eye T$ is a Hermitian matrix, and thus this latter matrix has real eigenvalues, whose absolute values are also the singular values of $T$.
We consider the performance of an algorithm thresholding the largest eigenvalue of this matrix (equivalently, the largest singular value of $T$), and find that its performance is inferior by a polynomial factor in $n$ in the required signal strength $\gamma$ compared to our algorithm based on a simple low-degree polynomial.
\begin{theorem}[Spectral detection thresholds]
    \label{thm:spectral}
    Suppose $\gamma = c\cdot n^{-1/2}$.
    Then, the following hold:
    \begin{enumerate}
        \item If $T \sim \sQ$, then $\frac{1}{\sqrt{n}}\lambda_{\max}(\eye T) \to 2$ in probability.
        \item If $T \sim \sP$ and $c \leq \pi / 4$, then $\frac{1}{\sqrt{n}}\lambda_{\max}(\eye T) \to 2$ in probability.
        \item If $T \sim \sP$ and $c > \pi / 4$, then $\frac{1}{\sqrt{n}}\lambda_{\max}(\eye T) \geq 2 + f(c)$ for some $f(c) > 0$ with high probability.
    \end{enumerate}
\end{theorem}
\noindent
In words, the result says that the success of a detection algorithm computing and thresholding the leading order term of $\lambda_{\max}(\eye T)$ undergoes a transition around the critical value $\gamma = \frac{\pi}{4}n^{-1/2}$, much greater than the scale $\gamma \sim n^{-3/4}$ for which the success of our low-degree polynomial algorithm undergoes the same transition.
The proof of this result relates $\eye T$ to a complex-valued \emph{spiked matrix model}, a low-rank additive perturbation of a Hermitian matrix of i.i.d.\ noise.

\begin{remark}
    We would like to point out that technically, Theorem \ref{thm:spectral} does not rule out the existence of a spectral algorithm that successfully distinguishes $\mathcal{P}$ and $\mathcal{Q}$ by thresholding $\lambda_{\max}(\eye T)$ when $\gamma$ is below $\frac{\pi}{4}n^{-1/2}$, since we only focused on the behavior of $\lambda_{\max}(\eye T)$ to leading order and ignored the smaller $o(\sqrt{n})$ fluctuations around the leading order term. For some $\gamma < \frac{\pi}{4}n^{-1/2}$, potentially there could exist $\epsilon = \epsilon(n, \gamma) > 0$ such that $\lambda_{\max}(\eye T) > (2 + \epsilon)\sqrt{n}$ w.h.p.~for $T \sim \mathcal{P}$ but $\lambda_{\max}(\eye T) < (2 + \epsilon)\sqrt{n}$ w.h.p.~for $T \sim \mathcal{Q}$, thus leaving open the possibility of the success of the spectral algorithm below the threshold mentioned in Theorem \ref{thm:spectral}; however, our results imply that such $\epsilon$ would need to have $\epsilon = o(1)$ as $n \to \infty$. We believe this is an interesting issue to address in future work.
\end{remark}

\subsection{Recovery}

For the recovery tasks under the planted model, we use the following standard metric on the set of permutations to measure how much error an estimator makes.
\begin{definition}[Kendall tau metric]
    For $\pi_1, \pi_2 \in S_n$, we define
    \begin{align}
        d(\pi_1, \pi_2) \colonequals \sum_{(i,j): i <_{\pi_1} j } \boldsymbol{1}\{i >_{\pi_2} j\}.
    \end{align}
\end{definition}
\noindent
This metric, known as the Kendall tau distance between two permutations, counts the number of pairs of inconsistent comparisons between the two permutations.
The goal of the recovery task is then to find a permutation that is close to the hidden permutation in this metric when given a tournament $T$ drawn from the planted distribution $\mathcal{P}$.

\begin{remark}[Spearman's footrule distance]
    Another commonly-used distance between permutations is given by
    \begin{equation}
        d^{\prime}(\pi_1, \pi_2) \colonequals \sum_{i = 1}^n |\pi_1(i) - \pi_2(i)|.
    \end{equation}
    This is known as the \emph{Spearman's footrule} or \emph{dislocation distance} between $\pi_1$ and $\pi_2$.
    In fact, $d$ and $d^{\prime}$ are equivalent up to constant factors \cite{DG-1977-SpearmanFootruleDisarray}, so all of our results for $d$ will also hold for $d^{\prime}$.
\end{remark}

We note that choosing a permutation from $S_n$ uniformly at random will always in expectation recover half of the total number of pairs of comparisons of the hidden permutation, achieving an expected Kendall tau distance from the hidden permutation of $\frac{1}{2}\binom{n}{2}$. Thus, we define the following notion of success of recovery for this model.

\begin{definition}[Weak and strong recovery]
    An algorithm $A$ that takes in a tournament $T$ on $n$ vertices and outputs a permutation $\hat{\pi} \in S_n$ is said to achieve \emph{weak recovery} of the planted permutation if there exists a constant $\delta > 0$ such that
    \begin{align*}
        \limsup_{n \to \infty} \frac{\mathbb{E}_{(T, \pi) \sim \mathcal{P}}[d(\pi, \hat{\pi})]}{\binom{n}{2} } \le \frac{1}{2} - \delta.
    \end{align*}
    It is said to achieve \emph{strong recovery} of the planted permutation if
    \begin{align*}
        \lim_{n \to \infty} \frac{\mathbb{E}_{(T, \pi) \sim \mathcal{P}}[d(\pi, \hat{\pi})]}{\binom{n}{2} } = 0.
    \end{align*}
\end{definition}
\noindent
Using these definitions, we may give a precise version of our results on recovery.

\begin{theorem}[Recovery thresholds] \label{thm:recovery-th}
    Suppose $0\le \gamma\le 1/4$.
    The following hold:
    \begin{enumerate}
    \item If $\gamma = \omega(n^{-1/2})$, then a polynomial-time algorithm achieves strong recovery of the planted permutation.

    \item If $\gamma = \Theta(n^{-1/2})$, then no algorithm achieves strong recovery of the planted permutation (i.e., strong recovery is information-theoretically impossible).

    \item If $\gamma = \Theta(n^{-1/2})$, then a polynomial-time algorithm achieves weak recovery of the planted permutation.

    \item If $\gamma = o(n^{-1/2})$, then no algorithm achieves weak recovery of the planted permutation (i.e., weak recovery is information-theoretically impossible).
    \end{enumerate}
\end{theorem}

\begin{remark}
    Notice that in Theorem \ref{thm:recovery-th} we assume $\gamma \le 1/4$. This arises as a technical requirement for us, but the case $\gamma \ge 1/4$ is covered by many previous works such as \cite{braverman2007noisy, klein2011tolerant, rubinstein2017sorting, geissmann2017sorting, geissmann2018optimal}. In particular, \cite{braverman2007noisy} give a polynomial-time algorithm that w.h.p.~produces a permutation at Kendall tau distance $O(n)$ from the hidden permutation for any constant $\gamma > 0$.
\end{remark}

\begin{remark}
    While we defined an algorithm achieving weak recovery as one that performs better than a random guess \textbf{in expectation} above, our positive result for weak recovery in Part 3 of Theorem~\ref{thm:recovery-th} is actually stronger, in the sense that our algorithm performs better than a random guess \textbf{with high probability}, which will be made clear from our proof in Section~\ref{sec:recovery:upper}.
\end{remark}

For our proofs of the negative results on recovery, we bound a distributional distance similar to our detection argument, in this case controlling the Kullback-Leibler (KL) divergence between the distributions of tournaments given by the planted model with different planted permutations (this approach may also be viewed as relating the feasibility of recovery to that of detection between two different planted models).
For the positive results, we show that the following surprisingly simple algorithm in fact performs optimally.

\begin{definition}[Ranking by wins]
    The \emph{Ranking By Wins} algorithm takes as input a tournament matrix $T$ and outputs a permutation $\hat{\pi} \in S_n$ in the following way:
    \begin{enumerate}
    \item For each $i\in [n]$, compute a score $s_i = \sum_{k\in [n]}T_{i,k}$,
    \item Rank the elements $i\in [n]$ according to the scores $s_i$ from the highest to the lowest, under an arbitrary tie-breaking rule (say, ranking $i$ below $j$ if $i < j$ when $s_i = s_j$).
    \end{enumerate}
\end{definition}
\noindent
If we think of each directed edge $(i,j)$ as representing that player $i$ won against player $j$ in a round-robin tournament, then $s_i = 2\cdot \# \{\text{wins of player } i\} - (n-1)$. Therefore, ranking according to the scores $s_i$ is the same as ranking according to the number of wins of each player, hence the name of the algorithm.

\subsection{Alignment Maximization and Maximum Likelihood Estimation}
\label{sec:mle}

Finally, we state our results on optimizing the alignment objective function, which we recall is given by
\begin{equation}
    \aln(\hat{\pi}, T) = \sum_{(i,j)\in E(T)} \big(\boldsymbol{1}\{\hat{\pi}(i) < \hat{\pi}(j)\} - \boldsymbol{1}\{\hat{\pi}(i) > \hat{\pi}(j)\}\big).
\end{equation}

Let us first draw a connection to maximum likelihood estimation to explain why it the alignment objective is an interesting objective function.
Let $T$ be a tournament drawn from the planted distribution $\mathcal{P}$.
The likelihood function $\mathcal{L}(\hat{\pi} \mid T)$ in this case can be expressed as
\begin{align*}
    \mathcal{L}(\hat{\pi} \mid T) &= \Px_{T \sim \mathcal{P}}[T \mid \hat{\pi}]\\
    &= \prod_{(i,j)\in E(T)}\left(\frac{1}{2} + \gamma\right)^{\boldsymbol{1}\{\hat{\pi}(i) < \hat{\pi}(j)\}} \left(\frac{1}{2} - \gamma\right)^{\boldsymbol{1}\{\hat{\pi}(i) > \hat{\pi}(j)\}}\\
    &= \left(\frac{1}{2} + \gamma\right)^{\sum_{(i,j)\in E(T)} \boldsymbol{1}\{\hat{\pi}(i) < \hat{\pi}(j)\} }\left(\frac{1}{2} - \gamma\right)^{\sum_{(i,j)\in E(T)} \boldsymbol{1}\{\hat{\pi}(i) > \hat{\pi}(j)\} }\\
    &= \left(\frac{1}{2} - \gamma\right)^{\frac{1}{2}\binom{n}{2}} \left(\frac{1}{2} + \gamma\right)^{\frac{1}{2}\binom{n}{2}} \left(\frac{\frac{1}{2} + \gamma}{\frac{1}{2} - \gamma}\right)^{\frac{1}{2}\cdot \aln(\hat{\pi}, T)},
\end{align*}
where the last line follows from $\sum_{(i,j)\in E(T)} \boldsymbol{1}\{\hat{\pi}(i) < \hat{\pi}(j)\} + \sum_{(i,j)\in E(T)} \boldsymbol{1}\{\hat{\pi}(i) > \hat{\pi}(j)\} = \binom{n}{2}$. Thus, the maximizer of the alignment objective has the pleasant statistical interpretation of being the maximum likelihood estimator of the hidden permutation under the planted distribution $\mathcal{P}$, given the observation $T$.
On the other hand, computing the maximum likelihood estimator or (equivalently) optimizing the alignment objective for a general worst-case input $T$ is NP-hard, as mentioned earlier.

Nevertheless, as our results below will show, when we consider draws from the planted model when strong recovery is information-theoretically possible, then the same simple Ranking By Wins algorithm closely approximates the MLE.

\begin{theorem}[Alignment maximization]\label{thm:alignment-approx}
    Suppose $0\le \gamma \le 1/4$. For $\gamma = \omega(n^{-1/2})$, there exists a polynomial-time algorithm which, given input a tournament $T$ drawn from $\mathcal{P}$, outputs a permutation $\hat{\pi}$ that w.h.p.~satisfies
    \begin{equation}
        \aln(\hat{\pi}, T) \ge (1 - o(1) ) \cdot \max_{\Tilde{\pi} \in S_n} \, \aln(\Tilde{\pi}, T).
    \end{equation}
\end{theorem}

\begin{remark}
    Again in Theorem \ref{thm:alignment-approx} we focus on the case $\gamma \leq 1/4$ for technical reasons. In the case of $\gamma \ge 1/4$, the maximum likelihood estimator can be computed exactly in polynomial time w.h.p.~\cite{braverman2007noisy}, which equivalently exactly maximizes the alignment objective.
\end{remark}

While this algorithm will be the same as the algorithm for estimating the hidden permutation under the Kendall tau distance, we emphasize that a good such estimator $\hat{\pi}$ does \emph{not} necessarily \emph{a priori} give a good approximate maximizer of the alignment objective.
Yet, it turns out that the Ranking By Wins estimator does have this property, which requires further probabilistic analysis of its behavior.

\section{Proofs of Detection Thresholds}

In this section, we prove Theorems~\ref{thm:detection-th} and \ref{thm:spectral}.
Recall that Theorem~\ref{thm:detection-th} has four parts; Parts 1 and 3 claim that polynomial-time algorithms for (weak and strong, respectively) detection exist, while Parts 2 and 4 claim that (weak and strong, respectively) detection is information-theoretically impossible.
It will turn out our treatments of weak and strong detection in each case are very closely related.
Thus, this section is organized first by giving some preliminaries in Section~\ref{sec:detection:prelim}, then the proofs of Parts 2 and 4 in Section~\ref{sec:detection:lower}, and finally the proofs of Parts 1 and 3 in Section~\ref{sec:detection:upper}.
We then prove Theorem~\ref{thm:spectral} in Section~\ref{sec:spectral}.

\subsection{Preliminaries}
\label{sec:detection:prelim}

We begin with some tools for performing Boolean Fourier analysis over tournaments.
\begin{definition}
    A \emph{shape} $S$ is a subset of the undirected edges of the complete graph $K_n$, with its vertices viewed as labelled by $[n]$.
    We will use $V(S) \subseteq [n]$ to denote the set of vertices that any edges of $S$ are incident to.
    Let $T \in \{\pm 1\}^{\binom{[n]}{2} }$ be a tournament matrix. We use $T^S$ to denote the monomial
    \begin{equation*}
        T^S \colonequals \prod_{\{i,j\} \in S: i < j} T_{i,j}.
    \end{equation*}
\end{definition}

Since under $T \sim \sQ$ the entries above the diagonal of $T$ are i.i.d.\ with entries drawn uniformly at random from $\{\pm 1\}$, the $T^S$ over all $S$ form an orthonormal basis of the space of real-valued functions of tournaments endowed with the inner product
\begin{equation}
    \langle f, g \rangle_{\mathcal{Q} } \colonequals \mathbb{E}_{\mathcal{Q}}[fg].
\end{equation}
This follows since the $T^S$ are just the suitable standard Boolean Fourier basis of dimension $\binom{n}{s}$.
We record some useful properties of this basis, starting with two results that give a direct proof of this orthonormality.

\begin{proposition}
    Let $S_1, S_2 \subseteq \binom{[n]}{2}$ be two shapes. For any tournament matrix $T$, we have $T^{S_1} T^{S_2} = T^{S_1 \triangle S_2}$.
\end{proposition}

\begin{proposition}
    Let $S \subseteq \binom{[n]}{2}$ be a non-empty shape. Then, $\mathbb{E}_{\mathcal{Q} }[T^S] = 0$.
\end{proposition}

\begin{corollary}[Orthonormality]
    For $S_1, S_2 \subseteq [n]$, if $f_i(T) = T^{S_i}$, then $\langle f_1, f_2\rangle_{\sQ} = \boldsymbol{1}\{S_1 = S_2\}$.
\end{corollary}

We will also need some tools for decomposing the likelihood ratio between $\sP$ and $\sQ$ in this basis, which, as we will see, amounts to evaluating the expectations of $T^S$ under $T \sim \sP$.

\begin{proposition}[Planted expectations] \label{prop:fourier-expectation}
    Let $S \subseteq \binom{[n]}{2}$ be a shape. Then,
    \begin{align*}
        \Ex_{T \sim \mathcal{P} }[T^S] &= (2\gamma)^{|S|} \Ex_{\pi \sim \mathrm{Unif}(S_n)}\left[(-1)^{\sum_{\{i,j\}\in S: i < j} \boldsymbol{1}\{\pi(i) > \pi(j)\} } \right].
    \end{align*}
\end{proposition}

\begin{proof}
    Recall that, to sample a tournament $T$ from $\mathcal{P}$, one may first sample a permutation $\pi \in S_n$ uniformly at random, and then generate a tournament $T$ that correlates suitably with $\pi$. For a fixed $\pi \in S_n$, let us denote by $\sP_{\pi}$ the distribution $\sP$ conditional on the hidden permutation being $\pi$. In particular, notice that $\mathcal{P}_{\pi}$ is a product distribution, where each directed edge of $T \sim \mathcal{P}_{\pi}$ is chosen independently between all pairs of $i,j$ (but with different distributions depending on $\pi$). Then, we have
    \begin{align*}
        \Ex_{T \sim \mathcal{P} }[T^S] &= \Ex_{\pi \sim \mathrm{Unif}(S_n)} \Ex_{T \sim \mathcal{P}_{\pi}}[T^S]\\
        &= \Ex_{\pi \sim \mathrm{Unif}(S_n)} \prod_{\{i,j\} \in S: i<j} \Ex_{T \sim \mathcal{P}_{\pi}}[T_{i,j}] \\
        &= \Ex_{\pi \sim \textrm{Unif}(S_n)} \prod_{\{i,j\} \in S: i<j} \left((-1)^{\boldsymbol{1}\{\pi(i)>\pi(j)\}} (2\gamma)\right) \\
        &= (2\gamma)^{|S|} \Ex_{\pi \sim \mathrm{Unif}(S_n)}\left[(-1)^{\sum_{\{i,j\}\in S: i < j} \boldsymbol{1}\{\pi(i) > \pi(j)\} } \right],
    \end{align*}
    completing the proof.
\end{proof}

\begin{proposition} \label{prop:component-independence}
    Let $S = S_1 \sqcup S_2$ be a shape with two vertex-disjoint components $S_1$ and $S_2$. Then,
    \begin{align*}
        \Ex_{T \sim \mathcal{P} }[T^S] = \Ex_{T \sim \mathcal{P} }[T^{S_1}]\Ex_{T \sim \mathcal{P} } [T^{S_2}].
    \end{align*}
\end{proposition}

\begin{proof}
    Since $S_1, S_2$ are vertex disjoint, the distribution of $T^{S_1}$ and $T^{S_2}$ under $\mathcal{P}$ are independent, as we can independently sample a permutation $\pi_1$ on the vertex set of $S_1$ and a permutation $\pi_2$ on the vertex set of $S_2$, and then sample the directed edges used in $S_1$ and $S_2$ which correlate with $\pi_1$ and $\pi_2$ respectively. Thus, $\mathbb{E}_{\mathcal{P} }[T^{S_1 \sqcup S_2}] = \mathbb{E}_{\mathcal{P} }[T^{S_1}] \mathbb{E}_{\mathcal{P} }[T^{S_2}]$.
\end{proof}

\begin{proposition} \label{prop:edge-decay}
    Let $S \subseteq \binom{[n]}{2}$ be a shape. Then, $\left|\mathbb{E}_{\mathcal{P} }[T^S]\right| \le (2\gamma)^{|S|}$.
\end{proposition}

\begin{proof}
    This follows directly from Proposition \ref{prop:fourier-expectation}.
\end{proof}

\begin{proposition} \label{prop:odd-cancellation}
    Let $S \subseteq \binom{[n]}{2}$ be a shape. If $S$ has an odd number of edges, then $\mathbb{E}_{\mathcal{P}}[T^S] = 0$.
\end{proposition}

\begin{proof}
    For any $\pi \in S_n$, we let $\text{rev}(\pi) \in S_n$ denote the reverse of $\pi$, given by $\text{rev}(\pi)(i) = n+1-\pi(i)$ for all $i\in [n]$.
    In particular, for any pair of $i,j$ such that $\pi(i) < \pi(j)$, we have $\text{rev}(\pi)(i)> \text{rev}(\pi)(j)$ and vice versa.
    Let us also write
    \[ \mathrm{swaps}(\pi) \colonequals \sum_{\{i,j\}\in S: i < j} \boldsymbol{1}\{\pi(i) > \pi(j)\}. \]
    By Proposition \ref{prop:fourier-expectation}, for any $S$ with an odd number of edges, we have
    \begin{align*}
        \mathbb{E}_{\mathcal{P} }[T^S]
        &= (2\gamma)^{|S|} \cdot \mathbb{E}_{\pi \sim \textrm{Unif}(S_n)}\left[(-1)^{\mathrm{swaps}(\pi)} \right] \\
        &= \frac{1}{2}(2\gamma)^{|S|}\cdot \mathbb{E}_{\pi \sim \textrm{Unif}(S_n)}\left[(-1)^{\mathrm{swaps}(\pi)} + (-1)^{\mathrm{swaps}(\mathrm{rev}(\pi))}\right] .
    \end{align*}
    For any fixed $\pi \in S_n$, we observe that
    \[ \mathrm{swaps}(\pi) + \mathrm{swaps}(\mathrm{rev}(\pi)) = |S|.\]
    Since $|S|$ is odd, for every $\pi \in S_n$, one of the quantities above is odd and the other is even. We thus find that $\EE_{\sP}[T^S] = 0$.
\end{proof}

Next, the specific shape of a path of length two will play an important role in the design of our polynomial-time detection algorithm.
We give an explicit calculation of the corresponding expectation.
\begin{proposition}\label{prop:wedge-expectation}
    $\mathbb{E}_{\mathcal{P} }[T_{1,2}T_{1,3}] = \frac{1}{3}(2\gamma)^2$.
\end{proposition}

\begin{proof}[Proof of Proposition \ref{prop:wedge-expectation}]
    We may use Proposition \ref{prop:fourier-expectation} to verify this equation.
\end{proof}

Finally, we will use the following result, a consequence of the Neyman-Pearson Lemma, on controlling error probabilities in detection or hypothesis testing using the total variation distance.
\begin{proposition}
    \label{prop:sum-type1-type2}
    Suppose $\sQ$ and $\sP$ are arbitrary probability measures (not necessarily our null and planted models) on some $\Omega$, such that $\sP$ is absolutely continuous with respect to $\sQ$, and let $f: \Omega \to \{0, 1\}$.
    Then,
    \begin{equation}
        \Px_{T \sim \sQ}[f(T) = 1] + \Px_{T \sim \sP}[f(T) = 0] \geq 1 - d_{\text{TV}}(\mathcal{P}, \mathcal{Q}).
    \end{equation}
\end{proposition}
\begin{proof}
    We may view $f(T)$ as a hypothesis test between $\sQ$ and $\sP$, which outputs 0 to indicate $\sQ$ and 1 to indicate $\sP$.
    Thus the expression on the left-hand side above is the sum of Type~I and Type~II error probabilities of this hypothesis test.
    By the Neyman-Pearson Lemma \cite[Lemma 5.3]{Rigollet-2017-HighDimensionalStatistics}, the test $f_{\eta}(T) \colonequals \boldsymbol{1}\{\frac{d\sP}{d\sQ}(T) \geq \eta\}$ minimizes the left-hand side for some choice of $\eta$.
    The minimum of the left-hand side evaluated on $f_{\eta}$ over all $\eta \in \RR$ is achieved at $\eta = 1$, and for this value the left-hand side equals $1 - d_{\text{TV}}(\mathcal{P}, \mathcal{Q})$.
\end{proof}

\subsection{Information-Theoretic Impossibility of Detection}
\label{sec:detection:lower}

We now proceed to our proof of the negative results included in Theorem~\ref{thm:detection-th}.

\begin{proof}[Proof of Parts 2 and 4 of Theorem~\ref{thm:detection-th}]
    We will establish an upper bound on the $\chi^2$-divergence between $\mathcal{P}$ and $\mathcal{Q}$ and use it to establish the impossibility of strong and weak detection.
    Let us review what conditions on the $\chi^2$-divergence would imply such impossibility results.

    When the $\chi^2$-divergence is bounded as $n \to \infty$, then $\sP$ is \emph{contiguous} to $\sQ$, whereby strong detection is impossible.
    See, e.g., \cite[Lemma 1.13]{kunisky2019notes} for details; this argument is sometimes called ``Le Cam's second moment method.''
    When the $\chi^2$-divergence is $o(1)$, using that the total variation distance is bounded as
    \[ d_{\text{TV}}(\mathcal{P}, \mathcal{Q}) \le \sqrt{\log(\chi^2(\mathcal{P} \dbar \mathcal{Q}) + 1)},\]
    we may show that weak detection is impossible.
    By Proposition~\ref{prop:sum-type1-type2}, the minimum sum of Type~I and Type~II error probabilities achievable by any algorithm for distinguishing $\mathcal{P}$ from $\mathcal{Q}$ is $1 - d_{\text{TV}}(\mathcal{P}, \mathcal{Q})$.

    Let us write $p$ and $q$ for the probability mass functions of $\sP$ and $\sQ$, respectively.
    The $\chi^2$-divergence between $\mathcal{P}$ and $\mathcal{Q}$, defined as
    \begin{align}
        \chi^2(\mathcal{P} \dbar \mathcal{Q}) = \sum_{T} \frac{p(T)^2}{q(T)} - 1,
    \end{align}
    where the summation is over all $n \times n$ tournament matrices $T$, can then be expressed as
    \begin{align*}
        \chi^2(\mathcal{P} \dbar \mathcal{Q}) &= \sum_{T} \frac{ p(T)^2}{q(T)} - 1\\
        &= \Ex_{T \sim \mathcal{Q}} \left(\frac{p(T)}{ q(T)}\right)^2 - 1
        \intertext{and, taking the Boolean Fourier transform,}
        &= \Ex_{T \sim \mathcal{Q}} \left(\sum_{S \subseteq \binom{[n]}{2}}\Ex_{T \sim \mathcal{Q}}\left[\frac{p(T)}{q(T)}\cdot T^S\right]\cdot T^S\right)^2 - 1\\
        &=  \sum_{S_1, S_2 \subseteq \binom{[n]}{2}}\mathbb{E}_{\mathcal{P}}\left[ T^{S_1}\right]\mathbb{E}_{\mathcal{P}}\left[ T^{S_2}\right]\cdot \mathbb{E}_{\mathcal{Q}}[T^{S_1} T^{S_2}] - 1\\
        &= \sum_{S \subseteq \binom{[n]}{2}} \left(\mathbb{E}_{\mathcal{P} }[T^S]\right)^2 - 1.
    \end{align*}

    By Proposition \ref{prop:component-independence} and Proposition \ref{prop:odd-cancellation}, the only terms that contribute to the sum above correspond to shapes $S$ that have an even number of edges in each connected component. Thus,
    {\allowdisplaybreaks
    \begin{align*}
        \chi^2(\mathcal{P} \dbar \mathcal{Q}) + 1 &= \sum_{S \subseteq \binom{[n]}{2} } \left(\mathbb{E}_{\mathcal{P} }[T^S]\right)^2 \\
        &= 1+\sum_{v=2}^n \sum_{\ell=1}^{\binom{n}{2}} \sum_{\substack{S \subseteq \binom{[n]}{2}\\ |V(S)|=v, |S|=\ell }} \left(\mathbb{E}_{\mathcal{P} }[T^S]\right)^2
        \intertext{where we separate out the term $S = \emptyset$ that contributes $(\mathbb{E}_{\mathcal{P}}[1])^2 = 1$,}
        &= 1+\sum_{\ell=2}^{\binom{n}{2}} \sum_{v = 2}^n  \sum_{\substack{S \subseteq \binom{[n]}{2}\\ |V(S)|=v, |S|=\ell\\
        \text{each connected component of } S\\ \text{ has an even number of edges} }} \left(\mathbb{E}_{\mathcal{P} }[T^S]\right)^2 \\
        &\le  1+\sum_{\ell=2}^{\binom{n}{2}} \sum_{2 \le v \le \min\{\frac{3}{2} \ell, n\} } \sum_{\substack{S \subseteq \binom{[n]}{2}\\ |V(S)|=v, |S|=\ell\\
        \text{each connected component of } S\\ \text{ has an even number of edges} }} (2\gamma)^{2\ell}
        \intertext{by Proposition \ref{prop:edge-decay}, and using that $|V(S)| \le |S| + \cc(S) \le \frac{3}{2}|S|$, where $\cc(S)$ denotes the number of connected components of $S$, and since each component has at least $2$ edges, $\cc(S) \le \frac{1}{2}|S|$,}
        &\le 1+\sum_{\ell=2}^{\binom{n}{2}} \sum_{2 \le v \le \min\{\frac{3}{2}\ell, n\} } \binom{n}{v} \binom{\binom{v}{2} }{\ell} (2\gamma)^{2\ell}\\
        &\le 1+\sum_{\ell=2}^{\binom{n}{2}} \sum_{2 \le v \le \min\{\frac{3}{2}\ell, n\}} \frac{n^v}{v!} \frac{\left(\frac{v^2}{2}\right)^\ell}{\ell!} (2\gamma)^{2\ell}\\
        &\le 1+\sum_{\ell=2}^{\binom{n}{2}} \frac{(4\gamma^2)^\ell}{2^\ell \ell!} \sum_{2 \le v \le \min\{\frac{3}{2}\ell, n\}} \frac{n^v v^{2\ell}}{v!}\\
        &\le 1+\sum_{\ell=2}^{\binom{n}{2}} \frac{(4\gamma^2)^\ell}{2^\ell \ell!} \sum_{2 \le v \le \min\{\frac{3}{2}\ell, n\}} \frac{e^v n^v v^{2\ell}}{v^v}\\
        \intertext{we notice that $\frac{d}{dv}\frac{e^v n^v v^{2\ell}}{v^v} = \frac{e^v n^v v^{2\ell}(\log n - \log v + 2\frac{\ell}{v})}{v^{v}} \ge 0$ for $1 \le v \le n$, so we can bound the inner sum by the number of terms times the maximum term, which is attained at $v = \min\{\frac{3}{2}\ell, n\}$,}
        &\le 1+\sum_{\ell=2}^{\binom{n}{2}} \frac{(4\gamma^2)^\ell}{2^\ell \ell!}\cdot z \frac{e^z n^z z^{2\ell}}{z^z}
        \intertext{where $z = \min\{\frac{3}{2}\ell, n\}$,}
        &\le 1+\sum_{\ell = 2}^{\frac{2}{3}n} \frac{(4\gamma^2)^\ell}{2^\ell \ell!}\cdot \frac{3}{2}\ell \frac{e^{\frac{3}{2}\ell } n^{\frac{3}{2}\ell } (\frac{3}{2}\ell)^{2\ell}}{(\frac{3}{2}\ell)^{\frac{3}{2}\ell}} + \sum_{l > \frac{2}{3}n} \frac{(4\gamma^2)^\ell}{2^\ell \ell!}\cdot n \frac{e^{n } n^{n } n^{2\ell}}{n^{n}}\\
        &\le 1+\sum_{\ell} \frac{3}{2}\ell \left(\frac{4\gamma^2 e^{\frac{5}{2}} n^{\frac{3}{2}} \cdot\sqrt{\frac{3}{2}} }{2\sqrt{\ell}}\right)^{\ell} + \sum_{\ell > \frac{2}{3}n} ne^n \left(\frac{4\gamma^2 e n^2}{2 \ell}\right)^{\ell}.
    \end{align*}
    }

    Observe that the sum above is bounded whenever $\gamma = O(n^{-3/4})$, and the sum is $1 + o(1)$ when $\gamma = o(n^{-3/4})$. As a result, if $\gamma = O(n^{-3/4})$, $\chi^2(\mathcal{P} \dbar \mathcal{Q}) = O(1)$ and thus strong detection is impossible. If $\gamma = o(n^{-3/4})$, then we have $\chi^2(\mathcal{P} \dbar \mathcal{Q}) = o(1)$ and thus weak detection is impossible.
\end{proof}

\subsection{Low-Degree Detection Algorithm}
\label{sec:detection:upper}

We now prove the positive results about the existence of polynomial-time detection algorithms included in Theorem~\ref{thm:detection-th}.
The proofs are motivated by the mechanics of the proofs of negative results above: we build a polynomial to use for detection by choosing a subset of Boolean Fourier modes that contribute a large amount to the $\chi^2$-divergence in the computation above when it diverges.

\begin{proof}[Proof of Parts 1 and 3 of Theorem~\ref{thm:detection-th}]
    Per the above motivation, we will consider the following degree $2$ polynomial in the input tournament matrix $T$:
    \begin{equation}
        f(T) = \sum_{i=1}^n \sum_{\substack{j,k \in [n] \setminus \{i\}\\j < k }} T_{i,j} T_{i,k}.
    \end{equation}
    Note that $f(T)$ is the sum of all Boolean Fourier modes whose shape is the path of length 2 (with any vertex labelling).
    We will show that there is a threshold $C$ such that $f(T) \leq C$ w.h.p.\ when $T \sim \sQ$, while $f(T) \geq C$ w.h.p.\ when $T \sim \sP$.
    To do this, we will compute the first two moments of $f(T)$ under each distribution and use Chebyshev's inequality.

    First, let us give some intuition about why $f(T)$ is a reasonable test statistic.
    For one interpretation, one may arrive at $f(T)$ by noticing that, when the $\chi^2$-divergence in our previous calculations diverges, then the summands of this shape make a divergent contribution.
    For another, more intuitive interpretation, we may rewrite
    \begin{equation}
        f(T) = K + \frac{1}{2}\sum_{i = 1}^n \left(\sum_{j \in [n] \setminus \{i\}} T_{i, j}\right)^2
    \end{equation}
    for some constant $K = K_n$.
    Therefore, up to translation and rescaling, $f(T)$ is the sample variance of the numbers of wins of various tournament vertices, or of the scores $s_i = \sum_{j \in [n] \setminus \{i\}} T_{i, j}$ from the definition of the Ranking By Wins algorithm.
    In other words, $f(T)$ will be larger when the distribution of win counts is more spread out, which we expect to occur under the planted model with sufficiently strong signal.

    Returning to our proof strategy, we now compute the first two moments of $f(T)$ under $\sQ$ and $\sP$.
    Under the null distribution $\mathcal{Q}$, we have
    \begin{align*}
        \mathbb{E}_{\mathcal{Q} }[f(T)] &= \sum_{i=1}^n \sum_{\substack{j,k \in [n] \setminus \{i\}\\j < k }} \mathbb{E}_{\mathcal{Q} }\left[T_{i,j} T_{i,k}\right]\\
        &= 0,\\
        \mathbb{E}_{\mathcal{Q} }[f(T)^2] &= \sum_{i_1, i_2 \in [n]} \sum_{\substack{j_1, k_1\in [n] \setminus \{i_1\}\\j_1 < k_1 } } \sum_{\substack{j_2, k_2\in [n] \setminus \{i_2\}\\j_2 < k_2 } }\mathbb{E}_{\mathcal{Q} }\left[T_{i_1, j_1} T_{i_1, k_1} T_{i_2, j_2} T_{i_2, k_2}\right]\\
        &= \sum_{i_1, i_2 \in [n]} \sum_{\substack{j_1, k_1\in [n] \setminus \{i_1\}\\j_1 < k_1 } } \sum_{\substack{j_2, k_2\in [n] \setminus \{i_2\}\\j_2 < k_2 }} \boldsymbol{1}\{i_1 = i_2, j_1 = j_2, k_1 = k_2\}\\
        &= \frac{n(n-1)(n-2)}{2}.
    \end{align*}
    Under the planted distribution $\mathcal{P}$, we have
    {\allowdisplaybreaks
    \begin{align*}
        &\mathbb{E}_{\mathcal{P} }[f(T)]\\
        &= \sum_{i=1}^n \sum_{j,k \in [n] \setminus\{i\} } \mathbb{E}_{\mathcal{P} }[T_{i,j}T_{i,k}]\\
        &= \frac{n(n-1)(n-2)}{2} \mathbb{E}_{\mathcal{P} }[T_{12}T_{13}]\\
    &= \binom{n}{3} \cdot (2\gamma)^2,
    \intertext{by Proposition \ref{prop:wedge-expectation}, and}
    &\Var_{\mathcal{P} }[f(T)]\\
    &= \mathbb{E}_{\mathcal{P} }[f(T)^2] - \mathbb{E}_{\mathcal{P} }[f(T)]^2\\
    &= \sum_{\substack{i_1, i_2 \in [n]\\
    j_1, k_1 \in [n] \setminus \{i_1\}, j_1 < k_1\\
    j_2, k_2 \in [n] \setminus \{i_2\}, j_2 < k_2 }} \Big(\mathbb{E}_{\mathcal{P} }[T_{i_1j_1}T_{i_1k_1} T_{i_2j_2}T_{i_2k_2}] -\mathbb{E}_{\mathcal{P} }[T_{i_1j_1}T_{i_1k_1}] \mathbb{E}_{\mathcal{P} }[T_{i_2j_2}T_{i_2k_2}]\Big)\\
    &\le \sum_{\substack{i_1, i_2 \in [n]\\
    j_1, k_1 \in [n] \setminus \{i_1\}, j_1 < k_1\\
    j_2, k_2 \in [n] \setminus \{i_2\}, j_2 < k_2\\
    \{i_1, j_1, k_1\} \cap \{i_2, j_2, k_2\} \ne \emptyset }} \mathbb{E}_{\mathcal{P} }[T_{i_1j_1}T_{i_1k_1} T_{i_2j_2}T_{i_2k_2}]
    \intertext{by Proposition \ref{prop:component-independence} and Proposition \ref{prop:wedge-expectation},}
    &\le \sum_{\substack{i_1, i_2 \in [n]\\
    j_1, k_1 \in [n] \setminus \{i_1\}, j_1 < k_1\\
    j_2, k_2 \in [n] \setminus \{i_2\}, j_2 < k_2\\
    |\{i_1, j_1, k_1\} \cap \{i_2, j_2, k_2\}| = 3  }} \mathbb{E}_{\mathcal{P} }[T_{i_1j_1}T_{i_1k_1} T_{i_2j_2}T_{i_2k_2}] + \sum_{\substack{i_1, i_2 \in [n]\\
    j_1, k_1 \in [n] \setminus \{i_1\}, j_1 < k_1\\
    j_2, k_2 \in [n] \setminus \{i_2\}, j_2 < k_2\\
    |\{i_1, j_1, k_1\} \cap \{i_2, j_2, k_2\}| = 2  }} \mathbb{E}_{\mathcal{P} }[T_{i_1j_1}T_{i_1k_1} T_{i_2j_2}T_{i_2k_2}]\\
    &+ \sum_{\substack{i_1, i_2 \in [n]\\
    j_1, k_1 \in [n] \setminus \{i_1\}, j_1 < k_1\\
    j_2, k_2 \in [n] \setminus \{i_2\}, j_2 < k_2\\
    |\{i_1, j_1, k_1\} \cap \{i_2, j_2, k_2\}| = 1  }} \mathbb{E}_{\mathcal{P} }[T_{i_1j_1}T_{i_1k_1} T_{i_2j_2}T_{i_2k_2}]\\
    &\le O(n^3 + n^4 \gamma^2 + n^5 \gamma^4),
    \end{align*}
    by Proposition \ref{prop:edge-decay}, where we use that there are $O(n^{6-t})$ pairs of $(\{i_1, j_1, k_1\}, \{i_2, j_2, k_2\})$ with $t$ shared vertices, and that if $|\{i_1, j_1, k_1\} \cap \{i_2, j_2, k_2\}| = t$, then the symmetric difference of $\{\{i_1, j_1\}, \{i_1, k_1\}\}$ and $\{\{i_2, j_2\}, \{i_2, k_2\}\}$ has at least $6 - 2t$ edges for $t \in \{1, 2, 3\}$.
    }

    By Chebyshev's inequality, whenever
    \begin{align*}
        \max\{\sqrt{\Var_{\mathcal{P} }[f(T)]}, \sqrt{\Var_{\mathcal{Q}}[f(T)]} \} = o\left(\mathbb{E}_{\mathcal{P} }[f(T)] - \mathbb{E}_{\mathcal{Q} }[f(T)] \right),
    \end{align*}
    then thresholding the value of $f(T)$ achieves strong detection between $\mathcal{P}$ and $\mathcal{Q}$.
    Moreover, if
    \begin{align*}
        (2\sqrt{2} + \varepsilon)\cdot \max\{\sqrt{\Var_{\mathcal{P} }[f(T)]}, \sqrt{\Var_{\mathcal{Q}}[f(T)]} \} \le \mathbb{E}_{\mathcal{P} }[f(T)] - \mathbb{E}_{\mathcal{Q} }[f(T)]
    \end{align*}
    for some constant $\varepsilon > 0$, then thresholding the value of $f(T)$ achieves weak detection between $\mathcal{P}$ and $\mathcal{Q}$.

    Plugging in the mean and variance bounds above, we get that strong detection is possible when $\gamma = \omega(n^{-3/4})$ and that for some constant $c > 0$, weak detection is possible when $\gamma \ge c \cdot n^{-\frac{3}{4}}$.
\end{proof}

\subsection{Suboptimality of Spectral Detection Algorithm}
\label{sec:spectral}

Finally, we consider the performance of a natural spectral algorithm for detection.
By ``spectral algorithm'' we mean one that takes $T$ as input, viewed as a tournament matrix, and outputs some data involving the top eigenvalues or singular values of a matrix constructed from $T$.

Since $T$ is skew-symmetric, it is \emph{a priori} unclear how best to design a spectral algorithm.
As stated in Theorem~\ref{thm:spectral}, we consider the following natural choice.
We will use bold $\eye$ when referring to the imaginary unit.
Instead of working with $T$, we work with $\eye T$, which is a Hermitian matrix and thus has real eigenvalues.
Those real eigenvalues are related to the singular values of $T$: the singular values of $T$ are the absolute values of the eigenvalues of $\eye T$.
They are also related to the eigenvalues of $T$: those eigenvalues were purely imaginary and came in conjugate pairs, and if $\eye \lambda$ and $-\eye \lambda$ are eigenvalues of $T$ then $\lambda$ and $-\lambda$ will be eigenvalues of $\eye T$.

We will consider a detection algorithm that computes and thresholds $\lambda_{\max}(\eye T)$ (which is the same as the spectral radius or operator norm of either $T$ or $\eye T$ by the above observation on the symmetry of the eigenvalues).
To carry out this analysis, we must understand this largest eigenvalue when $T \sim \sQ$ (Part 1 of Theorem~\ref{thm:spectral}) and when $T \sim \sP$ (Parts 2 and 3 of Theorem~\ref{thm:spectral}).

The former case is straightforward, since then $T$ is just a Hermitian Wigner matrix of i.i.d.\ (albeit complex) entries.
\begin{proof}[Proof of Part 1 of Theorem~\ref{thm:spectral}]
Recall that we want to show that, when $T \sim \sQ$, then
    \begin{equation}
        \frac{1}{\sqrt{n}} \lambda_{\max}(\eye T) \,\,\to\,\, 2,
    \end{equation}
    where the convergence is in probability.
    The entries of $\eye T$ above the diagonal are i.i.d.\ centered complex random variables whose modulus is always equal to 1.
    The result then follows by standard analysis of Wigner matrices; see, e.g., Sections 2.1.6 and 2.2 of \cite{AGZ-2010-RandomMatrices}.
\end{proof}

The latter case $T \sim \sP$ is more complicated.
Morally, it is similar to a \emph{spiked matrix model}, a low-rank additive perturbation of a Wigner matrix.
We will eventually appeal to the analysis of \cite{CDMF-2009-DeformedWigner} of general such models for perturbations of constant rank.
Yet, we will see that in this case the perturbation is actually not quite low-rank but rather has full rank with rapidly decaying eigenvalues, which must be treated more carefully.

In particular, we will decompose $\eye T = (\eye T - \EE \eye T) + \EE \eye T$, where we will see that the first term is a Wigner matrix, while the second term will be the additive perturbation whose effect we will understand by viewing this expression as a spiked matrix model.
We now compute this second term along with its eigenvalues.
Note without loss of generality that we may restrict all of our discussion to the case where the hidden permutation $\pi$ is the identity; we will condition on this event without further discussion going forward.
To be explicit about this, we will replace $\sP$ with $\sP_{\mathrm{id}}$ when we mention it.

Let $A \in \mathbb{C}^{n \times n}$ be defined as
    \begin{align}
        A_{i,j} = \begin{cases}
            \eye & \quad \text{ if } i < j, \\
            -\eye & \quad \text{ if } i > j, \\
            0 & \quad \text{ if } i = j.
        \end{cases}
    \end{align}
In words, $A$ is a Hermitian matrix with $\eye$ in the upper diagonal entries, $-\eye$ in the lower diagonal entries, and $0$ on the diagonal.
\begin{proposition}
    $\EE_{T \sim \sP_{\mathrm{id}}}[\eye T] = 2 \gamma A$.
\end{proposition}
\noindent
Next let us understand the spectrum of $A$.
    \begin{proposition}
        \label{prop:A-evecs}
        The eigenvalues of $A$ are given by
        \begin{align*}
            \lambda_i(A) = \frac{1}{\tan\left(\frac{2i-1}{2n}\pi\right)} \,\,\text{ for }\,\, 1\le i\le n
        \end{align*}
        with the corresponding eigenvectors $v_i \in \mathbb{C}^n$ whose entries are given by
        \begin{align*}
            (v_i)_j = \exp\left(-\eye \pi \frac{(2i-1)j}{n}\right) \,\,\text{ for }\,\, 1\le j\le n.
        \end{align*}
    \end{proposition}
    \begin{proof}
        We explicitly multiply the putative eigenvectors by $A$ to verify the claim:
        {\allowdisplaybreaks
        \begin{align*}
            (Av_i)_j
            &= \sum_{t=1}^n A_{j,t} (v_i)_t\\
            &= -\sum_{t=1}^{j-1} \eye \cdot \exp\left(-\eye \pi \frac{(2i-1)t}{n}\right)  + \sum_{t=j+1}^n \eye \cdot \exp\left(-\eye \pi \frac{(2i-1)t}{n}\right) \\
            &= \eye \left(-\sum_{t=1}^{j-1} \exp\left(-\eye \pi \frac{(2i-1)t}{n}\right)  + \sum_{t=j+1}^n \exp\left(-\eye \pi \frac{(2i-1)t}{n}\right)\right)\\
            &= \eye\left(\sum_{t=1}^{j-1} \exp\left(-\eye \pi \frac{(2i-1)t}{n} - \eye\pi (2i -  1)\right)  + \sum_{t=j+1}^n \exp\left(-\eye \pi \frac{(2i-1)t}{n}\right)\right)\\
            &= \eye\sum_{t= j+1}^{n+j-1} \exp\left(-\eye \pi \frac{(2i-1)t}{n}\right) \\
            &= \eye\cdot \exp\left(-\eye \pi \frac{(2i-1)j}{n}\right) \cdot \sum_{t = 1}^{n-1} \exp\left(-\eye \pi \frac{(2i-1)t}{n}\right)\\
            &= \eye\cdot \exp\left(-\eye \pi \frac{(2i-1)j}{n}\right) \cdot \exp\left(-\eye \pi \frac{2i-1}{n}\right) \cdot \frac{1 - \exp\left(-\eye \pi \frac{(2i-1)(n - 1)}{n}\right)}{1 - \exp\left(-\eye \pi \frac{2i - 1}{n}\right)}\\
            &= \exp\left(-\eye \pi \frac{(2i-1)j}{n}\right) \cdot \eye \cdot \frac{1 + \exp\left(-\eye \pi \frac{2i-1}{n}\right)}{1 - \exp\left(-\eye \pi \frac{2i-1}{n}\right)}\\
            &= \exp\left(-\eye \pi \frac{(2i-1)j}{n}\right) \cdot \eye \cdot \frac{\exp\left(\eye \pi \frac{2i-1}{2n}\right) + \exp\left(-\eye \pi \frac{2i-1}{2n}\right)}{\exp\left(\eye \pi \frac{2i-1}{2n}\right) - \exp\left(-\eye \pi \frac{2i-1}{2n}\right)}\\
            &= \exp\left(-\eye \pi \frac{(2i-1)j}{n}\right) \cdot \frac{1}{\tan\left(\frac{2i-1}{2n}\pi\right)}\\
            &= \frac{1}{\tan\left(\frac{2i-1}{2n}\pi\right)} \cdot (v_i)_j,
        \end{align*}
        so $v_i$ is an eigenvector corresponding to eigenvalue $\lambda_i = 1/\tan((2i-1)\pi/(2n))$. Moreover, since the eigenvalues $\lambda_i$ are distinct, the eigenvectors $v_i$ form an orthogonal basis of $\CC^n$, completing the proof.
        }
    \end{proof}

We may now give the analysis of the spectral algorithm on the planted model.
The interpretation of this result is that, for $\gamma = c \cdot n^{-1/2}$, for $c$ large enough there is at least one outlier eigenvalue greater than the typical largest eigenvalue under the null model, while for $c$ smaller the largest eigenvalue is the same as that of the null model.

\begin{proof}[Proof of Parts 2 and 3 of Theorem~\ref{thm:spectral}]
    Recall the statement of these results: when $T \sim \sP$ and $\gamma = c \cdot n^{-1/2}$ for a constant $c > 0$, we want to show that:
    \begin{enumerate}
        \item If $c \leq \pi / 4$, then
        \begin{equation}
            \frac{1}{\sqrt{n}} \lambda_{\max}(\eye T) \to 2
        \end{equation}
        in probability.
        \item If $c > \pi / 4$, then
        \begin{equation}
            \frac{1}{\sqrt{n}} \lambda_{\max}(\eye T) \geq 2 + f(c)
        \end{equation}
        for some $f(c) > 0$ with high probability.
    \end{enumerate}
    We may also assume without loss of generality that $T \sim \sP_{\mathrm{id}}$ rather than $T \sim \sP$.
    Let us write
    \begin{equation}
        \eye T = \EE \eye T + (\eye T - \EE \eye T) = 2\gamma A + (\eye T - \EE \eye T).
    \end{equation}
    We have that $W \colonequals \eye T - \EE \eye T$ is a Hermitian Wigner matrix whose entries above the diagonal are i.i.d.\ with the law of $\eye (X - 2\gamma) = \eye (X - \EE X)$ for $X \sim \mathrm{Rad}(\frac{1}{2} + \gamma)$.
    In particular, these entries are bounded, centered, and have complex variance $\EE|\eye (X - 2\gamma)|^2 = \Var(X) = 1 - 4\gamma^2 = 1 - O(1 / n)$ with our scaling of $\gamma$.

    We emphasize a small nuance: that $\eye T - \EE \eye T$ is exactly a Wigner matrix depends on our having assumed that the hidden permutation $\pi$ is the identity, since otherwise the entries above the diagonal would not be identically distributed.
    Yet, since conjugating a matrix by a permutation does not change the spectrum, the assumption that $\pi$ is any fixed permutation is without loss of generality; the choice of the identity permutation is just uniquely amenable to fitting into the existing theory around Wigner random matrices and spiked matrix models.

    Fix a large $k \in \NN$ not depending on $n$, to be chosen later.
    Let $\hat{v}_i \colonequals v_i / \|v_i\|$ for $v_i$ the eigenvectors of $A$ given in Proposition~\ref{prop:A-evecs}, so that $A = \sum_{i = 1}^n \lambda_i \hat{v}_i\hat{v}_i^{*}$.
    Let us write $A = A_1 + A_2$, where
    \begin{align}
        A_1 &\colonequals \sum_{i = 1}^k \lambda_i \hat{v}_i\hat{v}_i^{*} + \sum_{i = n - k + 1}^n \lambda_i \hat{v}_i \hat{v_i}^*, \\
        A_2 &\colonequals \sum_{i = k + 1}^{n - k} \lambda_i \hat{v}_i \hat{v_i}^*.
    \end{align}
    We have, since $\tan(x) \geq x / 2$ for sufficiently small $x$, that, for sufficiently large $n$,
    \begin{equation}
        \|A_2\| = \lambda_{k + 1} = \frac{1}{\tan(\frac{2k + 1}{2n} \pi)} \leq \frac{4n}{(2k + 1)\pi} = O\left(\frac{n}{k}\right).
    \end{equation}
    Also, the eigenvalues of $A_1$ satisfy, as $n \to \infty$, the convergences
    \begin{align}
        \frac{\lambda_a}{n} &= \frac{1}{n \tan(\frac{2a - 1}{2n}\pi)} \,\,\to\,\, \frac{2}{(2a - 1)\pi}, \\
        \frac{\lambda_{n - a}}{n} &= -\frac{\lambda_a}{n} \,\,\to\,\, -\frac{2}{(2a - 1)\pi}
    \end{align}
    for any $a$ fixed as $n \to \infty$.
    Thus we may define
    \begin{equation}
        A_0 \colonequals \sum_{i = 1}^k \frac{2n}{(2i - 1)\pi} \hat{v}_i\hat{v}_i^{*} - \sum_{i = 1}^k \frac{2n}{(2i - 1)\pi} \hat{v}_{n - i + 1} \hat{v}_{n - i + 1}^*,
    \end{equation}
    and this will satisfy $\|A_1 - A_0\| = o(n)$ as $n \to \infty$.

    We have
    \begin{equation}
        \label{eq:iT-decomp}
        \frac{1}{\sqrt{n}}\eye T = \left(\frac{2c}{n} A_0 + \frac{1}{\sqrt{n}}W\right) + \frac{2c}{n}(A_1 - A_0) + \frac{2c}{n} A_2.
    \end{equation}
    The rank of $A_0$ is $2k$, a constant as $n \to \infty$, the non-zero eigenvalues of $\frac{2c}{n}A_0$ are independent of $n$, and $W$ is a Wigner matrix with entrywise complex variance converging to 1.
    Thus, we may apply Theorem 2.1 of \cite{CDMF-2009-DeformedWigner} to the first term of \eqref{eq:iT-decomp}, which is a finite-rank additive perturbation of a Wigner matrix.
    That result implies that the largest eigenvalue of the first term converges to 2 in probability if no eigenvalue of $\frac{2c}{n}A_0$ is greater than 1, and converges to some $2 + f(c, k)$ with $f(c, k) > 0$ in probability if some eigenvalue of $\frac{2c}{n}A_0$ is greater than 1.
    We have $\lambda_{\max}(\frac{2c}{n}A_0) = \frac{4}{\pi}c$, so the latter condition is equivalently $c > \frac{\pi}{4}$.
    Moreover, since $\lambda_{\max}(\frac{2c}{n}A_0)$ depends only on $c$ and not on $k$, the details of the result of \cite{CDMF-2009-DeformedWigner} imply that $f(c, k) = f(c)$ depends only on $c$ and not on $k$ as well.
    Let us define this value as
    \begin{equation}
        \Lambda = \Lambda(c) \colonequals \begin{cases} 2 & \text{if } c \leq \frac{\pi}{4}, \\ 2 + f(c) & \text{if } c > \frac{\pi}{4}. \end{cases}
    \end{equation}

    It remains to show that the other two terms of \eqref{eq:iT-decomp} do not change this behavior substantially.
    We have $\|\frac{2c}{n}(A_1 - A_0)\| = o(1)$ and $\|\frac{2c}{n}A_2\| = O(\frac{1}{k})$, so by taking $k$ and $n$ sufficiently large we may make $\|\frac{2c}{n}(A_1 - A_0)\| + \|\frac{2c}{n}A_2\|$ smaller than any given $\eta > 0$ not depending on $n$.
    By the Weyl eigenvalue inequality, we then have $|\lambda_{\max}(\frac{1}{\sqrt{n}} \eye T) - \Lambda| \leq 2\eta$ with high probability.
    Since this is true for any $\eta > 0$, we also have that $\lambda_{\max}(\frac{1}{\sqrt{n}} \eye T) \to \Lambda$ in probability, and the result follows.
\end{proof}

\section{Tools for Analysis of Ranking By Wins Algorithm} \label{sec:ranking-by-wins}

In preparation for our results on recovery and alignment maximization, we next introduce some tools that will be useful in analyzing the Ranking By Wins algorithm.
In either case, the analysis will boil down to estimating the expected error or value achieved by the algorithm as well as controlling the fluctuations of this quantity.

To bound the fluctuation of solution output by the Ranking By Wins algorithm, we will use the following results on tail bounds for weakly dependent random variables.

\begin{definition}[Read-$k$ families \cite{gavinsky2015tail}]
    Let $X_1, \dots, X_m$ be independent random variables. Let $Y_1, \dots, Y_n$ be Boolean random variables such that $Y_j = f_j((X_i)_{i \in P_j})$ for some Boolean functions $f_j$ and index sets $P_j \subseteq [m]$. If the index sets satisfy $|\{j: i\in P_j\}|\le k$ for every $i\in [n]$, we say that $\{Y_j\}_{j=1}^n$ forms a \emph{read-$k$ family}.
\end{definition}

\begin{theorem}[Tail bounds for read-$k$ families \cite{gavinsky2015tail}]\label{thm:tail-read-k}
    Let $Y_1, \dots, Y_r$ be a read-$k$ family of Boolean random variables.
    Write $\mu \colonequals \EE \sum_{i = 1}^r Y_i$.
    Then, for any $t\ge 0$,
    \begin{align*}
        \mathbb{P}\left[\sum_{i=1}^r Y_i \ge \mu + t \right] &\le \exp\left(-\frac{2t^2}{rk}\right),\\
        \mathbb{P}\left[\sum_{i=1}^r Y_i \le \mu - t \right] &\le \exp\left(-\frac{2t^2}{rk}\right).
    \end{align*}
\end{theorem}

To estimate the expectation of the error or alignment objective value achieved by the Ranking By Wins algorithm, we will use the following version of the Berry-Esseen quantitative central limit theorem.

\begin{theorem}[Berry-Esseen theorem for non-identically distributed summands \cite{berry1941accuracy}]\label{thm:berry-esseen}
    Let $X_1, \dots, X_n$ be independent random variables with $\mathbb{E}[X_i] = 0, \mathbb{E}[X_i^2] = \sigma_i^2$, and $\mathbb{E}[|X_i|^3] = \rho_i < \infty$. Let
    \[S_n = \frac{\sum_{i=1}^n X_i}{\sqrt{\sum_{i=1}^n \sigma_i^2}}.\]
    Then, there exists an absolute constant $C>0$ independent of $n$ such that for any $x \in \mathbb{R}$,
    \begin{align*}
        \left|\mathbb{P}\left[S_n \le x\right] - \Phi(x)\right| \le C\cdot \frac{\max_{1\le i\le n} \frac{\rho_i}{\sigma_i^2}}{\sqrt{\sum_{i=1}^n \sigma_i^2}},
    \end{align*}
    where $\Phi: \mathbb{R} \to [0,1]$ is the cumulative distribution function (cdf) of the standard normal distribution.
\end{theorem}

After applying the Berry-Esseen theorem above, naturally we need to deal with expressions involving $\Phi$, the cdf of the standard normal distribution. Next we state a useful lemma for bounding certain sums involving the function $\Phi$.

\begin{lemma} \label{lem:concavity-Phi-expr}
    Let $a, b \ge 0$. As a function of $y$,
    \begin{align*}
        (1 - y) \cdot \Phi(-ay-b)
    \end{align*}
    is concave for $y\in [0,1]$.
\end{lemma}

\begin{proof}[Proof of Lemma \ref{lem:concavity-Phi-expr}]
    We compute the first and the second derivative of $(1 - y)\Phi(-ay-b)$.
    \begin{align*}
    \frac{d}{dy} (1 - y) \Phi\left(-ay-b\right)
    &= \frac{d}{dy} (1 - y) \int_{-\infty}^{-ay-b} \frac{1}{\sqrt{2\pi}} e^{-\frac{1}{2}z^2}dz\\
    &= - \int_{-\infty}^{-ay-b} \frac{1}{\sqrt{2\pi}} e^{-\frac{1}{2}z^2}dz + (1-y)\frac{1}{\sqrt{2\pi}} e^{-\frac{(ay+b)^2}{2}},\\
    \frac{d^2}{dy^2} (1 - y) \Phi\left(- ay-b\right)
    &= \frac{d}{dy}\left[- \int_{-\infty}^{-ay-b} \frac{1}{\sqrt{2\pi}} e^{-\frac{1}{2}z^2}dz + (1-y)\frac{1}{\sqrt{2\pi}} e^{-\frac{(ay+b)^2}{2}}\right]\\
    &= -\frac{1}{\sqrt{2\pi}} e^{-\frac{(ay+b)^2}{2}} - \frac{1}{\sqrt{2\pi}} e^{-\frac{(ay+b)^2}{2}} + (1-y) \frac{1}{\sqrt{2\pi}}\left(-a(ay+b)\right)e^{-\frac{(ay+b)^2}{2}}\\
    &= \frac{1}{\sqrt{2\pi}}e^{-\frac{(ay+b)^2}{2}}\left(a(y-1)(ay+b) - 2\right).
\end{align*}
We observe that the second derivative is negative for $y \in [0,1]$. Thus, $(1 - y)\Phi(-ay-b)$ is concave on $[0,1]$.
\end{proof}

\section{Proofs of Recovery Thresholds}

In this section, we prove Theorem~\ref{thm:recovery-th}.
As for our detection results, recall that the Theorem has four parts; Parts 1 and 3 claim that polynomial-time algorithms for (weak and strong, respectively) recovery exist, while Parts 2 and 4 claim that (weak and strong, respectively) recovery is information-theoretically impossible.
Again, this section proceeds by giving the proofs of Parts 2 and 4 in Section~\ref{sec:recovery:lower}, and then the proofs of Parts 1 and 3 in Section~\ref{sec:recovery:upper}.

\subsection{Information-Theoretic Impossibility of Recovery}
\label{sec:recovery:lower}

We now give our proof of the negative results included in Theorem~\ref{thm:recovery-th}.
We note that, since we are working only with the planted model for recovery results, if not otherwise indicated all expectations and probabilities are over $\sP$.

\begin{proof}[Proof of Parts 2 and 4 of Theorem~\ref{thm:recovery-th}]
    Consider any recovery algorithm, given by a function $A: \{\pm 1\}^{\binom{n}{2}} \to S_n$.
    In expectation, $A$ achieves an error of
    \begin{align*}
        \mathbb{E}[d(A(T), \pi)] &= \sum_{i<j} \mathbb{E}\left[\boldsymbol{1}\{A(T)(i,j)\ne \pi(i,j)\}\right]\\
        &= \sum_{i<j}\mathbb{P}\left[A(T)(i,j)\ne \pi(i,j)\right]. \numberthis \label{eq:kendall-tau-expectation}
    \end{align*}
    Thus, if for every $i < j$ we can show that $\mathbb{P}[A(T)(i,j)\ne \pi(i,j)]$ is bounded away from $0$ (or even close to $\frac{1}{2}$) for any function $A$, we will get an information-theoretic lower bound for the recovery problem.

    Let us rewrite
    \begin{equation}
        \Px_{(T, \pi) \sim \mathcal{P}}\left[A(T)(i,j)\ne \pi(i,j)\right] = \frac{1}{n!} \sum_{\pi \in S_n} \Px_{T \sim \mathcal{P}_{\pi}}\left[A(T)(i,j)\ne \pi(i,j)\right],
    \end{equation}
    where $\mathcal{P}_{\pi}$ denotes the distribution $\sP$ conditional on the hidden permutation being $\pi$. For every $\pi \in S_n$, let $\pi^{\{i,j\}} \in S_n$ denote the permutation obtained from $\pi$ by swapping $\pi(i)$ and $\pi(j)$. Then,
    \begin{align*}
        &\frac{1}{n!} \sum_{\pi \in S_n} \Px_{T \sim \mathcal{P}_{\pi}}\left[A(T)(i,j)\ne \pi(i,j)\right]\\
        &= \frac{1}{2\cdot n!} \sum_{\pi \in S_n} \left(\Px_{T \sim \mathcal{P}_{\pi}}\left[A(T)(i,j)\ne \pi(i,j)\right] + \Px_{T \sim \mathcal{P}_{\pi^{\{i,j\}}}}\left[A(T)(i,j)\ne \pi^{\{i,j\}}(i,j)\right]\right). \numberthis \label{eq:sum-of-testing-errors}
    \end{align*}
    Our plan is then to show that each summand is small.

    For every fixed $\pi \in S_n$, we note that the sum
    \[\Px_{T \sim \mathcal{P}_{\pi}}\left[A(T)(i,j)\ne \pi(i,j)\right] + \Px_{T \sim \mathcal{P}_{\pi^{\{i,j\}}}}\left[A(T)(i,j)\ne \pi^{\{i,j\}}(i,j)\right]\]
    may be viewed as the sum of Type-I and Type-II errors of the test statistic $T \mapsto A(T)(i,j) \in \{ \pm 1\}$ for distinguishing $\mathcal{P}_{\pi}$ and $\mathcal{P}_{\pi^{\{i,j\}}}$.
    By Proposition~\ref{prop:sum-type1-type2}, we have the lower bound
    \begin{align*}
        &\mathbb{P}_{T \sim \mathcal{P}_{\pi}}\left[A(T)(i,j)\ne \pi(i,j)\right] + \mathbb{P}_{T \sim \mathcal{P}_{\pi^{\{i,j\}}}}\left[A(T)(i,j)\ne \pi^{\{i,j\}}(i,j)\right] \ge 1 - d_{\text{TV}}\left(\mathcal{P}_{\pi}, \mathcal{P}_{\pi^{\{i,j\}}}\right).
    \end{align*}
    It remains to understand $d_{\text{TV}}(\mathcal{P}_{\pi}, \mathcal{P}_{\pi^{\{i,j\}}})$.

    We will bound the total variation distance by the KL divergence. Note that both distributions $\mathcal{P}_{\pi}$ and $\mathcal{P}_{\pi^{\{i,j\}}}$ are the product distributions of independent entries of the upper diagonal of the tournament matrix. Using the tensorization (additivity under concatenating independent random variables) of KL divergence, we have
    \begin{align*}
        D_{\text{KL}}\left(\mathcal{P}_{\pi}, \mathcal{P}_{\pi^{\{i,j\}}}\right)
        &= \sum_{a<b} D_{\text{KL}}\left(\mathcal{P}_{\pi}(T_{a,b}), \mathcal{P}_{\pi^{\{i,j\}}}(T_{a,b})\right).
    \end{align*}
    We note that if $a,b \in [n] \setminus \{i,j\}$, then the associated distributions of $T_{a,b}$ are identical for $\mathcal{P}_{\pi}$ and $\mathcal{P}_{\pi^{\{i,j\}}}$. Thus, the number of pairs of $(a,b)$ for which $D_{\text{KL}}\left(\mathcal{P}_{\pi}(T_{a,b}), \mathcal{P}_{\pi^{\{i,j\}}}(T_{a,b})\right)$ is nonzero is at most $2n$. Lastly, among those pairs of $(a,b)$ that have different distributions of $T_{a,b}$ under $\mathcal{P}_{\pi}$ and $\mathcal{P}_{\pi^{\{i,j\}}}$, each has KL divergence
    \begin{align*}
        D_{\text{KL}}\left(\mathcal{P}_{\pi}(T_{a,b}), \mathcal{P}_{\pi^{\{i,j\}}}(T_{a,b})\right)
        &= D_{\text{KL}}\left(\text{Rad}\left(\frac{1}{2} + \gamma\right), \text{Rad}\left(\frac{1}{2} - \gamma\right)\right),
    \end{align*}
    where we use the symmetry $D_{\text{KL}}\left(\text{Rad}\left(\frac{1}{2} + \gamma\right), \text{Rad}\left(\frac{1}{2} - \gamma\right)\right) = D_{\text{KL}}\left(\text{Rad}\left(\frac{1}{2} - \gamma\right), \text{Rad}\left(\frac{1}{2} + \gamma\right)\right)$. It is an easy calculus computation to show that
    \begin{align*}
        \quad D_{\text{KL}}\left(\text{Rad}\left(\frac{1}{2} + \gamma\right), \text{Rad}\left(\frac{1}{2} - \gamma\right)\right)
        &= \left(\frac{1}{2} + \gamma\right) \log \frac{\frac{1}{2} + \gamma}{\frac{1}{2} - \gamma} + \left(\frac{1}{2} - \gamma\right) \log \frac{\frac{1}{2} - \gamma}{\frac{1}{2} + \gamma}
        \intertext{and since $\log(1 + x) \le x$, we have}
        &\le \left(\frac{1}{2} + \gamma\right) \frac{2\gamma}{\frac{1}{2}-\gamma} - \left(\frac{1}{2} - \gamma\right) \frac{2\gamma}{\frac{1}{2}+\gamma}\\
        &= \frac{4\gamma^2}{\frac{1}{4} - \gamma^2}.
    \end{align*}
    We conclude that we have the bound
    \begin{align*}
        D_{\text{KL}}\left(\mathcal{P}_{\pi}, \mathcal{P}_{\pi^{\{i,j\}}}\right) \le 2n \cdot \frac{4\gamma^2}{\frac{1}{4} - \gamma^2} = \frac{8 n\gamma^2}{\frac{1}{4} - \gamma^2}.
    \end{align*}

    {\allowdisplaybreaks
    Finally, by Pinsker's inequality and its refinement for small values of the KL divergence (sometimes called the Bretagnolle-Huber inequality; see \cite[Equation 2.25]{Tsybakov_2009} and \cite{canonne2022short}), we have
    \begin{align*}
        &\mathbb{P}_{T \sim \mathcal{P}_{\pi}}\left[A(T)(i,j)\ne \pi(i,j)\right] + \mathbb{P}_{T \sim \mathcal{P}_{\pi^{\{i,j\}}}}\left[A(T)(i,j)\ne \pi^{\{i,j\}}(i,j)\right]\\
        &\ge 1 - d_{\text{TV}}\left(\mathcal{P}_{\pi}, \mathcal{P}_{\pi^{\{i,j\}}}\right)\\
        &\ge 1 - \min\left\{\sqrt{2D_{\text{KL}}\left(\mathcal{P}_{\pi}, \mathcal{P}_{\pi^{\{i,j\}}}\right)}, \sqrt{1 - \exp(-D_{\text{KL}}\left(\mathcal{P}_{\pi}, \mathcal{P}_{\pi^{\{i,j\}}}\right))}\right\}\\
        &\ge 1 - \min\left\{\sqrt{\frac{16 n\gamma^2}{\frac{1}{4} - \gamma^2}}, \sqrt{1 - \exp\left(-\frac{8 n\gamma^2}{\frac{1}{4} - \gamma^2}\right)}\right\}\\
        &\ge \max\left\{1 - \frac{4\sqrt{n}\cdot \gamma}{\sqrt{\frac{1}{4} - \gamma^2}}, \frac{1}{2} \exp\left(-\frac{8 n\gamma^2}{\frac{1}{4} - \gamma^2}\right) \right\}. \numberthis \label{ineq:KL-lb}
    \end{align*}
    Plugging \eqref{ineq:KL-lb} back into \eqref{eq:sum-of-testing-errors},
    \begin{align*}
        &\mathbb{P}_{(T, \pi) \sim \mathcal{P}}
\left[A(T)(i,j)\ne \pi(i,j)\right]\\
        &= \frac{1}{2n!} \sum_{\pi \in S_n} \left(\mathbb{P}_{T \sim \mathcal{P}_{\pi}}\left[A(T)(i,j)\ne \pi(i,j)\right] + \mathbb{P}_{T \sim \mathcal{P}_{\pi^{\{i,j\}}}}\left[A(T)(i,j)\ne \pi^{\{i,j\}}(i,j)\right]\right)\\
        &\ge \frac{1}{2}\max\left\{1 - \frac{4\sqrt{n}\cdot \gamma}{\sqrt{\frac{1}{4} - \gamma^2}}, \frac{1}{2} \exp\left(-\frac{8 n\gamma^2}{\frac{1}{4} - \gamma^2}\right)\right\}.
    \end{align*}
    Substituting it into \eqref{eq:kendall-tau-expectation}, we then obtain
    \begin{align*}
        \mathbb{E}[d(A(T), \pi)]
        &\ge \frac{1}{2}\binom{n}{2}\max\left\{1 - \frac{4\sqrt{n}\cdot \gamma}{\sqrt{\frac{1}{4} - \gamma^2}}, \frac{1}{2} \exp\left(-\frac{8 n\gamma^2}{\frac{1}{4} - \gamma^2}\right)\right\}.
    \end{align*}}

    We therefore conclude that if $\gamma = O(n^{-1/2})$, it is impossible to achieve strong recovery, and if $\gamma = o(n^{-1/2})$, it is impossible to achieve weak recovery.
\end{proof}

\subsection{Ranking By Wins Recovery Algorithm}
\label{sec:recovery:upper}

Next, we analyze the Ranking By Wins algorithm to show that it achieves weak or strong recovery whenever each is information-theoretically possible.

Recall that, in the Ranking By Wins algorithm, a permutation $\hat{\pi}$ is created by ranking according to the scores $s_i = \sum_{k \in [n]} T_{i,k}$.
Without loss of generality, we may assume that the hidden permutation $\pi$ is the identity.
Note that $\hat{\pi}$ is only well-defined up to the tie-breaking rule.
We will get rid of the technicality of the tie-breaking rules by employing a pessimistic view that the algorithm breaks ties in the worst possible way, i.e., that for every $i < j$, if $s_i = s_j$, then the algorithm ranks the pair $i,j$ incorrectly.
Thus, this upper bounds the Kendall tau distance $d(\pi, \hat{\pi})$ (achieved with by the Ranking By Wins algorithm with any tie-breaking rule) by a function $f(T)$ defined as
\begin{align*}
    f(T) \colonequals \sum_{i < j} \boldsymbol{1}\{s_i \le s_j\}.
\end{align*}
To proceed with the proof, we follow the outline sketched in Section \ref{sec:ranking-by-wins}, bounding the expectation of $f(T)$ and showing that it concentrates around its expectation.

\begin{proof}[Proof of Parts 1 and 3 of Theorem~\ref{thm:recovery-th}]
    While the random variables $\boldsymbol{1}\{s_i > s_j\}$ are not independent, they are only weakly dependent and form a read-$(2n)$ family. By Theorem \ref{thm:tail-read-k}, we get
    \begin{align*}
        \mathbb{P}\left[f(T) \ge \mathbb{E}[f(T)] + t\right] &\le \exp\left(-\frac{2t^2 }{2n \binom{n}{2}}\right),
    \end{align*}
    for any $t\ge 0$. Thus, we see that, with high probability,
    \begin{align}
        f(T) \le \mathbb{E}[f(T)]  + t(n) \label{ineq:recovery-tail-bd}
    \end{align}
    for any function $t(n) = \omega(n^{3/2})$.
    Now it remains to upper bound $\mathbb{E}[f(T)]$. We have
    \begin{align*}
        \mathbb{E}[f(T)]&= \sum_{i < j} \mathbb{P}[s_i \le s_j].
    \end{align*}
    For every pair $i < j$, we have
    \begin{align*}
        s_i - s_j &= \sum_{k \in [n]} T_{i,k} - \sum_{k \in [n]} T_{j,k}\\
        &\stackrel{(d)}{=} \sum_{t = 1}^{n-i+j-3} X_t + \sum_{t=1}^{n+i-j-1} Y_t + Z,
    \end{align*}
    where $X_t \stackrel{i.i.d}{\sim} \text{Rad}\left(\frac{1}{2} + \gamma\right)$, $Y_t \stackrel{i.i.d}{\sim} \text{Rad}\left(\frac{1}{2} - \gamma\right)$, and $Z \sim 2\cdot \text{Rad}\left(\frac{1}{2} + \gamma\right)$.

    Recall that we have assumed $\gamma \le \frac{1}{4}$.
    With this assumption, we have
    \begin{align*}
        \Var(X_t) = \Var(Y_t) = \frac{1}{4}\Var(Z) &= 1 - 4\gamma^2 \ge \frac{3}{4},\\
        \mathbb{E}[|X_t - \mathbb{E}[X_t]|^3] = \mathbb{E}[|Y_t - \mathbb{E}[Y_t]|^3] = \frac{1}{8}\mathbb{E}[|Z - \mathbb{E}[Z]|^3] &=1 - 16\gamma^4,\\
        \mathbb{E}\left[\sum_{t = 1}^{n-i+j-3} X_t + \sum_{t=1}^{n+i-j-1} Y_t + Z\right] &= 4(j-i)\gamma,\\
        \Var\left(\sum_{t = 1}^{n-i+j-3} X_t + \sum_{t=1}^{n+i-j-1} Y_t + Z\right) &= 2(1-4\gamma^2)n.
    \end{align*}
    Then, we apply Theorem \ref{thm:berry-esseen} and obtain that
    \begin{align*}
        &\mathbb{P}\left[s_i \le s_j\right]\\
        &= \mathbb{P}\left[\sum_{t = 1}^{n-i+j-3} X_t + \sum_{t=1}^{n+i-j-1} Y_t + Z \le 0\right]\\
        &= \mathbb{P}\left[\frac{\sum_{t = 1}^{n-i+j-3} X_t + \sum_{t=1}^{n+i-j-1} Y_t + Z - 4(j-i)\gamma}{\sqrt{2(1-4\gamma^2)n}} \le -\frac{4(j-i)\gamma}{\sqrt{2(1-4\gamma^2)n}}\right] \\
        &\le \Phi\left(- \frac{4(j-i)\gamma}{\sqrt{2(1-4\gamma^2)n}}\right) + C \cdot \frac{1}{\sqrt{n}}, \numberthis \label{ineq:berry-esseen-bd}
    \end{align*}
    for some absolute constant $C > 0$.

    Plugging this bound back to $\mathbb{E}[f(T)]$, we obtain
    \begin{align*}
    \mathbb{E}[f(T)] &\le \sum_{i < j} \left(\Phi\left(- \frac{4(j-i)\gamma}{\sqrt{2(1-4\gamma^2)n}}\right) + C \cdot\frac{1}{\sqrt{n}}\right)\\
    &= \sum_{i<j} \Phi\left(- \frac{4(j-i)\gamma}{\sqrt{2(1-4\gamma^2)n}}\right) + O\left(n^{\frac{3}{2}}\right)\\
    &= \sum_{d=1}^{n-1}(n-d)\Phi\left(- \frac{4d\gamma}{\sqrt{2(1-4\gamma^2)n}}\right) + O\left(n^{\frac{3}{2}}\right). \numberthis \label{ineq:recovery-expectation-ub}
\end{align*}
By Lemma \ref{lem:concavity-Phi-expr}, we may bound the sum above as
\begin{align*}
    &\sum_{d=1}^{n-1}(n-d)\Phi\left(- \frac{4d\gamma}{\sqrt{2(1-4\gamma^2)n}}\right)\\
    &= n\sum_{d=1}^{n-1}\left(1 - \frac{d}{n}\right)\Phi\left(-\frac{4\gamma\sqrt{n}\cdot \frac{d}{n}}{\sqrt{2(1-4\gamma^2)}}\right) \numberthis \label{eq:riemann-sum}
    \intertext{now, we use that $(1-y)\Phi
\left(-\frac{4\gamma\sqrt{n}\cdot y}{\sqrt{2(1-4\gamma^2)}}\right)$ is concave for $y\in [0,1]$, and we apply Jensen's inequality to get}
    &\le \binom{n}{2}\Phi\left(-\frac{2\gamma\sqrt{n}}{\sqrt{2(1-4\gamma^2)}}\right). \numberthis \label{ineq:sum-Phi-bd}
\end{align*}
Plugging it back to \eqref{ineq:recovery-expectation-ub}, we obtain
\begin{align}
    \mathbb{E}[f(T)] \le  \binom{n}{2}\Phi\left(-\frac{2\gamma\sqrt{n}}{\sqrt{2(1-4\gamma^2)}}\right) + O\left(n^{\frac{3}{2}}\right).
\end{align}

Together with \eqref{ineq:recovery-tail-bd}, for any $t(n) = \omega(n^{3/2})$, the permutation $\hat{\pi}$ output by the Ranking By Wins algorithm with high probability achieves a Kendall tau distance from the hidden permutation of at most
\begin{align}
    d(\pi, \hat{\pi}) \le f(T) \le \binom{n}{2}\Phi\left(-\frac{2\gamma\sqrt{n}}{\sqrt{2(1-4\gamma^2)}}\right) + t(n).
\end{align}

Qualitatively, this means that for any constant $c> 0$ and $\gamma = c\cdot n^{-1/2}$, there exists some constant $\delta = \delta(c) > 0$ such that the Ranking By Wins algorithm outputs a permutation that w.h.p.~achieves a Kendall tau distance of at most $\left(\frac{1}{2} - \delta \right) \binom{n}{2}$ from the hidden permutation, thus achieving weak recovery.
Moreover, as $c \to \infty$, this value satisfies $\delta = \delta(c) \to \frac{1}{2}$, whereby for any $\gamma = \omega(n^{-1/2})$, the Ranking By Wins algorithm achieves strong recovery of the hidden permutation.
\end{proof}

Let us remark on some more precise statements about the strong recovery regime.
Quantitatively, we may bound the cdf of the standard normal distribution by Mill's inequality \cite{gordon1941values} and get
\begin{align*}
    \binom{n}{2}\Phi\left(-\frac{2\gamma\sqrt{n}}{\sqrt{2(1-4\gamma^2)}}\right) &\le \binom{n}{2} \cdot  \frac{C}{\frac{2\gamma \sqrt{n}}{\sqrt{2(1-4\gamma^2)}}} \cdot e^{- \frac{\gamma^2 n}{1-4\gamma^2}}\\
    &\le \binom{n}{2} \frac{C}{\gamma\sqrt{n}} \cdot e^{-\gamma^2 n},
\end{align*}
for some constant $C > 0$. Thus, the Ranking By Wins algorithm with high probability achieves a Kendall tau distance from the hidden permutation of $O(\frac{n^{3/2}}{\gamma}\cdot e^{-\gamma^2 n}) + t(n)$ for any $t(n) = \omega(n^{3/2})$.
So, our positive result for the Ranking By Wins algorithm can prove at best an upper bound scaling slightly faster than $n^{3/2}$ on the Kendall tau error.
It seems likely that this could be improved, since we are pessimistically assuming that the errors from the Berry-Esseen bound used accumulate additively in \eqref{ineq:berry-esseen-bd}, while probably in reality they themselves enjoy some cancellations.

In the weak recovery regime, based on the intermediate Riemann sum result in \eqref{eq:riemann-sum}, one may also show the more precise asymptotic that, for $\gamma = c \cdot n^{-1/2}$, we have
\begin{equation}
    \frac{1}{n^2}\EE[f(T)] = \int_0^1 (1 - y) \Phi\left(-2\sqrt{2} \, c y\right)dy + o(1),
\end{equation}
which, with our remaining calculations, describes the precise amount of error incurred by the Ranking By Wins algorithm to leading order, with high probability.

\section{Proof of Approximation of Maximum Alignment}

We now show that the Ranking By Wins algorithm also approximately maximizes the alignment objective.

Before proving Theorem~\ref{thm:alignment-approx}, we state a high probability bound on the maximum alignment objective when $T$ is drawn from $\mathcal{P}$, to which we will compare the alignment objective achieved by Ranking By Wins.

\begin{proposition}\label{prop:opt-high-prob-bound}
    For a tournament $T$ drawn from $\mathcal{P}$, the optimum alignment objective with high probability satisfies
    \begin{align*}
        2\gamma \binom{n}{2} - \Tilde{O}(n) \le \max_{\hat{\pi} \in S_n} \, \aln(\hat{\pi}, T) \le 2\gamma \binom{n}{2} + O\left(n^{\frac{3}{2}}\right).
    \end{align*}
\end{proposition}

The proof of Proposition \ref{prop:opt-high-prob-bound} can be adapted from \cite{de1983maximum}, but we include here the main argument for the sake of completeness.

\begin{proof}[Proof of Proposition \ref{prop:opt-high-prob-bound}]
    Let us write $\mathrm{OPT} \colonequals \max_{\hat{\pi} \in S_n} \, \aln(\hat{\pi}, T)$.
    First, we prove the high probability lower bound on $\text{OPT}$. In fact, by Chernoff bound, we may show that the hidden permutation $\pi$ would achieve this lower bound w.h.p., as
    \begin{align*}
        \mathbb{P}\left[\sum_{i<j} \pi(i,j)T_{i,j} - \mathbb{E}\left[\sum_{i<j} \pi(i,j)T_{i,j}\right] \le -\lambda\right] \le \exp\left(-\frac{\lambda^2}{4\binom{n}{2}}\right)
    \end{align*}
    since for each $i< j$, $\pi(i,j)T_{i,j}$ is independently distributed as $\text{Rad}\left(\frac{1}{2} + \gamma\right)$. Further note that the expectation is $2\gamma\binom{n}{2}$. Plugging in $\lambda = C\cdot n\log n$, we get that w.h.p.~for $T$ drawn from $\mathcal{P}$,
    \[\text{OPT} \ge \sum_{i < j}\pi(i,j)T_{i,j} \ge 2\gamma\binom{n}{2} - C\cdot n\log n.\]

    Next we turn to a high probability upper bound on $\text{OPT}$.
    Without loss of generality, we assume that the hidden permutation is the identity. For convenience, we furthermore assume that $n = 2^k$ is a power of $2$. In the end we will easily see how to extend the analysis to arbitrary values of $n$. Following the argument in \cite{de1983maximum}, for any permutation $\hat{\pi} \in S_n$, we partition the set of $\binom{n}{2}$ directed edges of a tournament on $n$ vertices according to $\hat{\pi}$ into
    \begin{align*}
        \bigsqcup_{\ell=1}^k B_{\ell},
    \end{align*}
    where $B_{\ell}$ consists of the pairs of indices that belong to
    \begin{align*}
        B_{\ell} &= \Bigg\{(x,y) \in [n]^2: \text{there exists } 0\le i \le 2^{\ell-1}-1 \text{ such that }2i \frac{n}{2^l} + 1\le \hat{\pi}(x) \le (2i+1) \frac{n}{2^{\ell}},\\
        &\hspace{3.1cm} (2i+1) \frac{n}{2^{\ell}} + 1\le \hat{\pi}(y) \le (2i+2) \frac{n}{2^{\ell}}\Bigg\}.
    \end{align*}

    Then, we may write
    \begin{align*}
        \sum_{i<j} \hat{\pi}(i,j) T_{i,j} = \sum_{\ell=1}^k \sum_{(x,y) \in B_{\ell}} T_{x,y},
    \end{align*}
    and
    \begin{align*}
        \max_{\hat{\pi}} \sum_{i<j} \hat{\pi}(i,j) T_{i,j} \le \sum_{\ell=1}^k \max_{B_{\ell}} \sum_{(x,y) \in B_{\ell}} T_{x,y}.
    \end{align*}

    We then follow the argument in \cite{de1983maximum} and use a union bound over all $B_{\ell}$ for each $\ell \in [k]$. For each fixed $B_{\ell}$, we observe that the distribution of $\sum_{(x,y)\in B_{\ell} } T_{x,y}$ is stochastically dominated by that of a sum of $\frac{n}{2^{\ell}}$ independent Rademacher random variables $\text{Rad}\left(\frac{1}{2} + \gamma\right)$. Recycling the calculations from \cite{de1983maximum}, we can show that w.h.p.~for $T$ drawn from $\mathcal{P}$,
    \begin{align*}
        \max_{\hat{\pi}} \sum_{i<j} \hat{\pi}(i,j) T_{i,j} \le \sum_{\ell=1}^k \max_{B_{\ell}} \sum_{(x,y) \in B_{\ell}} T_{x,y} \le 2\gamma \binom{n}{2} + C\cdot n^{\frac{3}{2}},
    \end{align*}
    for some constant $C > 0$. Finally, for the values of $n$ that are not a power of $2$, we use a similar decomposition of $\binom{[n]}{2}$ into $B_{\ell}$, but we no longer decompose $[n]$ into $2^l$ subsets of equal sizes. Despite this technicality, the same idea works here \emph{mutatis mutandis}.
\end{proof}

We remark that there is also a worst-case lower bound of $\Omega(n^{3/2})$ on the optimum alignment objective for an arbitrary tournament on $n$ vertices, which was shown in \cite{spencer1971optimal}. Moreover, it is shown that one can construct in polynomial time a solution that always achieves this lower bound asymptotically in \cite{poljak1988tournament}.

\begin{theorem}[\cite{spencer1971optimal, poljak1988tournament}]\label{thm:opt-worst-case-lb}
    There exists a constant $C > 0$, such that for any tournament $T$ on $n$ vertices for all sufficiently large $n$,
    \begin{align}
        \max_{\hat{\pi} \in S_n} \, \aln(\hat{\pi}, T) \ge C \cdot n^{\frac{3}{2}},
    \end{align}
    and there exists a polynomial-time algorithm that outputs, given a tournament $T$, a $\hat{\pi} \in S_n$ having $\aln(\hat{\pi}, T) = \Omega(n^{3/2})$.
\end{theorem}

\begin{remark}
    Theorem \ref{thm:opt-worst-case-lb} above together with Proposition \ref{prop:opt-high-prob-bound} implies that when $\gamma = O(n^{-1/2})$, then the polynomial-time algorithm mentioned in Theorem \ref{thm:opt-worst-case-lb} w.h.p.~achieves a constant factor approximation to the maximum alignment on input $T$ drawn from $\mathcal{P}$.
    But, our Theorem~\ref{thm:alignment-approx} shows a near optimal approximation for larger $\gamma$, in which case the guarantee of Theorem~\ref{thm:opt-worst-case-lb} is sub-optimal.
\end{remark}

In what follows, we will thus focus on the regime when $\gamma = \omega(n^{-1/2})$. Our main result in this section is that when $\gamma = \omega(n^{-1/2})$, the Ranking By Wins algorithm w.h.p.~achieves a $(1-o(1))$ approximation to the maximum alignment on input $T$ drawn from $\mathcal{P}$.

\subsection{Alignment Objective Achieved by Ranking By Wins Algorithm}

In this section, we show that the Ranking By Wins algorithm achieves a good quality solution to the alignment objective for $T$ drawn from $\mathcal{P}$. Without loss of generality, we may assume that the hidden permutation in the planted model is the identity.

\begin{proof}[Proof of Theorem \ref{thm:alignment-approx}]
Recall the alignment objective of a tournament $T$ evaluated at a permutation $\hat{\pi}$ is equal to
\begin{align*}
    \aln(\hat{\pi}, T) = \sum_{i < j} \hat{\pi}(i,j) T_{i,j}.
\end{align*}
The permutation $\hat{\pi}$ output by Ranking By Wins algorithm satisfies
\begin{align*}
    \hat{\pi}(i,j) = \begin{cases}
        +1 & \text{ if } s_i > s_j, \\
        -1 & \text{ if } s_i < s_j,
    \end{cases}
\end{align*}
and when $s_i = s_j$, the Ranking By Wins algorithm break ties between $i$ and $j$ according to some tie-breaking rule. Regardless of the tie-breaking rule, we may always lower bound the alignment objective achieved by the Ranking By Wins algorithm with
{\allowdisplaybreaks
\begin{align*}
    &\sum_{i < j} \hat{\pi}(i,j) T_{i,j}\\
    &\ge \sum_{i < j} \left(\boldsymbol{1}\{s_i > s_j\} - \boldsymbol{1}\{s_i < s_j\}\right)T_{i,j} - \sum_{i<j} \boldsymbol{1}\{s_i = s_j\}\\
    &= \sum_{i < j} \left(\boldsymbol{1}\{s_i > s_j\} - \boldsymbol{1}\{s_i < s_j\}\right)(2\cdot b_{i,j} - 1) - \sum_{i<j} \boldsymbol{1}\{s_i = s_j\}\\
    &= 2\sum_{i<j} \boldsymbol{1}\{s_i > s_j\} \cdot b_{i,j} + \sum_{i<j} \boldsymbol{1}\{s_i < s_j\} - 2\sum_{i<j} \boldsymbol{1}\{s_i < s_j\} \cdot b_{i,j} - \sum_{i<j} \boldsymbol{1}\{s_i \ge s_j\}\\
    &\equalscolon f(T),
\end{align*}
where we introduce random variables $b_{i,j} = \frac{T_{i,j}+1}{2}$, so that $b_{i,j}$ are independently distributed as $\text{Ber}\left(\frac{1}{2} + \gamma\right)$ for $i < j$, and use $f(T)$ to denote this lower bound. It is easy to see that $\{\boldsymbol{1}\{s_i > s_j\} \cdot b_{i,j}: 1\le i<j \le n\}$, $\{\boldsymbol{1}\{s_i < s_j\}: 1\le i<j \le n\}$, $\{\boldsymbol{1}\{s_i < s_j\} \cdot b_{i,j}: 1\le i<j \le n\}$, $\{\boldsymbol{1}\{s_i \ge s_j\}: 1\le i<j \le n\}$ each form a read-$(2n)$ family. Thus, by Theorem \ref{thm:tail-read-k}, we may get that w.h.p.,}
\begin{align}
    &\quad f(T) \ge \mathbb{E}\left[f(T)\right] - t(n), \label{ineq:alignment-tail-bd}
\end{align}
for any function $t(n) = \omega(n^{3/2})$.
Now it remains to lower bound the expectation $\mathbb{E}\left[f(T)\right]$. We have
\begin{align*}
    &\mathbb{E}[f(T)]\\
    &= \mathbb{E}\left[2\sum_{i<j} \boldsymbol{1}\{s_i > s_j\} \cdot b_{i,j} + \sum_{i<j} \boldsymbol{1}\{s_i < s_j\} - 2\sum_{i<j} \boldsymbol{1}\{s_i < s_j\} \cdot b_{i,j} - \sum_{i<j} \boldsymbol{1}\{s_i \ge s_j\}\right]\\
    &= 2\sum_{i<j}\mathbb{P}[s_i > s_j, b_{i,j} = 1] + \sum_
{i<j} \mathbb{P}[s_i < s_j] - 2\sum_{i<j}\mathbb{P}[s_i < s_j, b_{i,j} = 1] - \sum_{i<j} \mathbb{P}[s_i \ge s_j]. \numberthis \label{eq:alignment-expectation}
\end{align*}
Notice that
\begin{align*}
    \mathbb{P}[s_i > s_j, b_{i,j} = 1] &= \mathbb{P}[s_i^{(j)} > s_j^{(i)} - 2] \cdot \mathbb{P}[b_{i,j} = 1]\\
    &= \left(\frac{1}{2} + \gamma\right)\mathbb{P}[s_i^{(j)} > s_j^{(i)} - 2], \numberthis \label{eq:Pij'-1}\\
    \mathbb{P}[s_i < s_j, b_{i,j} = 1] &= \mathbb{P}[s_i^{(j)} < s_j^{(i)} - 2]\cdot \mathbb{P}[b_{i,j} = 1]\\
    &= \left(\frac{1}{2} + \gamma\right)\mathbb{P}[s_i^{(j)} < s_j^{(i)} - 2], \numberthis \label{eq:Pij'-2}
\end{align*}
where we define $s_i^{(j)} = \sum_{k\in [n]: k\ne j} T_{i,k}$ and similarly define $s_j^{(i)}$. Moreover, since $s_i^{(j)} < s_j^{(i)} - 2$ implies $s_i < s_j$, and $s_i > s_j$ implies $s_i^{(j)} > s_j^{(i)} - 2$, we have
\begin{align}
    \mathbb{P}[s_i^{(j)} < s_j^{(i)} - 2] &\le \mathbb{P}[s_i < s_j],\\
    \mathbb{P}[s_i > s_j] &\le \mathbb{P}[s_i^{(j)} > s_j^{(i)} - 2].
\end{align}
Thus, together with \eqref{eq:alignment-expectation}, \eqref{eq:Pij'-1} and \eqref{eq:Pij'-2}, we get
\begin{align*}
    &\mathbb{E}[f(T)]\\
    &= 2\sum_{i<j}\mathbb{P}[s_i > s_j, b_{i,j} = 1] + \sum_{i<j} \mathbb{P}[s_i < s_j] - 2\sum_{i<j}\mathbb{P}[s_i < s_j, b_{i,j} = 1] - \sum_{i<j} \mathbb{P}[s_i \ge s_j]\\
    &= \sum_{i<j}\left(\left(1 + 2\gamma\right)\mathbb{P}[s_i^{(j)} > s_j^{(i)} - 2] - \mathbb{P}[s_i \ge s_j]\right) - \sum_{i<j}\left(\left(1+2\gamma\right)\mathbb{P}[s_i^{(j)} < s_j^{(i)} - 2] -\mathbb{P}[s_i < s_j] \right)\\
    &\ge \sum_{i<j}\left(\left(1 + 2\gamma\right)\mathbb{P}[s_i > s_j] - \mathbb{P}[s_i \ge s_j]\right) - \sum_{i<j}\left(\left(1+2\gamma\right)\mathbb{P}[s_i < s_j] -\mathbb{P}[s_i < s_j] \right)\\
    &= 2\gamma \sum_{i<j}\mathbb{P}[s_i > s_j] - 2\gamma\sum_{i<j}\mathbb{P}[s_i < s_j] - \sum_{i<j}\mathbb{P}[s_i = s_j].
\end{align*}

Then, we may apply Theorem \ref{thm:berry-esseen} to obtain
\begin{align*}
        \mathbb{P}\left[s_i > s_j\right] &\ge 1 - \Phi\left(- \frac{4(j-i)\gamma}{\sqrt{2(1-4\gamma^2)n}}\right) - C \cdot \frac{1}{\sqrt{n}}\\
        \mathbb{P}\left[s_i < s_j\right] &\le \Phi\left(- \frac{4(j-i)\gamma}{\sqrt{2(1-4\gamma^2)n}}\right) + C \cdot \frac{1}{\sqrt{n}}\\
        \mathbb{P}\left[s_i = s_j\right] &\le C \cdot \frac{1}{\sqrt{n}},
    \end{align*}
    for some constant $C > 0$ similarly as in \eqref{ineq:berry-esseen-bd} (with a different $C$ from that in \eqref{ineq:berry-esseen-bd}). We then further lower bound $\mathbb{E}[f(T)]$ by
    \begin{align*}
        \mathbb{E}[f(T)]
        &\ge 2\gamma \sum_{i<j}\mathbb{P}[s_i > s_j] - 2\gamma\sum_{i<j}\mathbb{P}[s_i < s_j] - \sum_{i<j}\mathbb{P}[s_i = s_j]\\
        &\ge 2\gamma\binom{n}{2} - 4\gamma \sum_{i<j} \Phi\left(- \frac{4(j-i)\gamma}{\sqrt{2(1-4\gamma^2)n}}\right) - O\left(n^{\frac{3}{2}}\right).
    \end{align*}

    Finally, we reuse the calculations in \eqref{ineq:sum-Phi-bd} and get{\allowdisplaybreaks
    \begin{align*}
        \mathbb{E}[f(T)]
        &\ge 2\gamma\binom{n}{2} - 4\gamma \sum_{i<j} \Phi\left(- \frac{4(j-i)\gamma}{\sqrt{2(1-4\gamma^2)n}}\right) - O\left(n^{\frac{3}{2}}\right)\\
        &\ge 2\gamma\binom{n}{2} - 4\gamma \binom{n}{2}\Phi\left(-\frac{2\gamma\sqrt{n}}{\sqrt{2(1-4\gamma^2)}}\right) - O\left(n^{\frac{3}{2}}\right)\\
        &= 2\gamma\binom{n}{2}\left(1 - 2\cdot \Phi\left(-\frac{2\gamma\sqrt{n}}{\sqrt{2(1-4\gamma^2)}}\right)\right) - O\left(n^{\frac{3}{2}}\right).
    \end{align*}
    Notice that in particular, if $\gamma = \omega(n^{-1/2})$, we have
    }
    \begin{align}
        \mathbb{E}[f(T)] \ge (1 - o(1))\cdot  2\gamma\binom{n}{2}. \label{ineq:alignment-expectation-lb}
    \end{align}

    Combining the fluctuation bound \eqref{ineq:alignment-tail-bd} and the expectation lower bound \eqref{ineq:alignment-expectation-lb}, we obtain that, for $\gamma = \omega(n^{-1/2})$, w.h.p.~the Ranking By Wins algorithm achieves an alignment objective of at least
    \begin{align*}
        \sum_{i<j} \hat{\pi}(i,j) T_{i,j} \ge f(T) \ge (1 - o(1)) \cdot 2\gamma \binom{n}{2}.
    \end{align*}
    On the other hand, by Proposition \ref{prop:opt-high-prob-bound}, for $\gamma = \omega(n^{-1/2})$, w.h.p.~the maximum alignment objective for a tournament $T$ drawn from $\mathcal{P}$ is at most
    \begin{align*}
        \text{OPT} \le 2\gamma\binom{n}{2} + O\left(n^{\frac{3}{2}}\right) = (1 + o(1)) \cdot 2\gamma\binom{n}{2}.
    \end{align*}
    Thus, we conclude that for $\gamma = \omega(n^{-1/2})$, w.h.p.~the Ranking By Wins algorithm achieves a $(1 - o(1))$ approximation of the maximum alignment on input $T$ drawn from $\mathcal{P}$.
\end{proof}

\addcontentsline{toc}{section}{References}
\bibliographystyle{alpha}
\bibliography{main}

\end{document}